\documentclass[10pt,reqno]{amsart}
\usepackage{graphicx,color}
\usepackage{amssymb,amscd,latexsym, mathrsfs, epsfig}
\usepackage[all]{xy}
\usepackage{todonotes}

\title[DT Frobenius type, CV-structures and convergence]{Frobenius type and CV-structures for Donaldson-Thomas theory and a convergence property}

\author[A. Barbieri]{Anna Barbieri}
\author[J. Stoppa]{Jacopo Stoppa}

  \address{Universit\`a di Pavia, Dipartimento di Matematica ``F. Casorati", Via A. Ferrata 1, 27100 Pavia, Italy}
  
  \email{anna.barbieri01@universitadipavia.it}
  
  \email{jacopo.stoppa@unipv.it}
  
\thanks{}

\usepackage{hyperref,url}
\usepackage{amssymb,latexsym}
\usepackage{amsthm,amssymb,amscd}
\usepackage[latin1]{inputenc}
\usepackage{enumerate}
\usepackage{todonotes}
\usepackage[all]{xy} \CompileMatrices
\SelectTips{cm}{12} 

\usepackage{verbatim} 

\theoremstyle{plain}
\newtheorem{theorem}{Theorem}[section]
\newtheorem{lem}[theorem]{Lemma}
\newtheorem{corollary}[theorem]{Corollary}
\newtheorem{proposition}[theorem]{Proposition}

\theoremstyle{definition}
\newtheorem{definition}[theorem]{Definition}
\newtheorem{definition-theorem}[theorem]{Definition-Theorem}

\theoremstyle{remark}
\newtheorem{remark}[theorem]{Remark}

\numberwithin{equation}{section} 

\newcommand{\pow}[1]{[\![ {#1} ]\!]}

\newcommand\PP{\mathbb P}
\newcommand\C{\mathbb C}
\newcommand\Q{\mathbb Q}
\newcommand\R{\mathbb R}
\newcommand\Z{\mathbb Z}

\newcommand{\z}{\zeta}

\newcommand\Hom{\operatorname{Hom}}
\newcommand\Aut{\operatorname{Aut}}
\renewcommand\Im{\operatorname{Im}}

\newcommand{\bra}{\langle}
\newcommand{\ket}{\rangle}
\newcommand{\wed}{\wedge}
\newcommand{\del}{\partial}

\newcommand{\B}{\mathcal{B}}

\newcommand{\M}{\mathcal{M}}
\newcommand{\T}{T}

\newcommand{\dt}{\operatorname{DT}}

\newcommand{\Stab}{\operatorname{Stab}}

\renewcommand{\th}{\theta}

\newcommand{\g}{\mathfrak{g}}

\newcommand{\pv}{\operatorname{pv}}

\makeatletter \@addtoreset{equation}{section} \makeatother

\newcommand{\cA}{\mathcal{A}}
\newcommand{\cC}{\mathcal{C}}

\newcommand{\cQ}{\mathcal{Q}}

\newcommand{\cU}{\mathcal{U}}
\newcommand{\cV}{\mathcal{V}}

\newcommand{\ad}{\operatorname{ad}}
\newcommand{\Ad}{\operatorname{Ad}}

\begin{document}

\begin{abstract} We rephrase some well-known results in Donaldson-Thomas theory in terms of (formal families of) Frobenius type and CV-structures on a vector bundle in the sense of Hertling. We study these structures in an abstract setting, and prove a convergence result which is relevant to the case of triangulated categories. An application to physical field theory is also briefly discussed.
\end{abstract} 

\maketitle 
\setcounter{tocdepth}{1}
\tableofcontents
\section{Introduction}

\subsection{Formal infinite-dimensional picture} Frobenius manifolds are (complex or real) manifolds endowed with a special structure on their tangent bundle. They were introduced by Dubrovin (see e.g. \cite{dubrovin}) and play a key role in quantum cohomology, the enumerative theory of rational curves on algebraic varieties. 

There is a notion of Frobenius type structure on a general holomorphic bundle, due to Hertling \cite{hert}.  A particular Frobenius type structure on an auxiliary infinite-dimensional bundle plays an important role in Donaldson-Thomas theory, the enumerative theory of semistable objects in abelian and triangulated categories. This is essentially a rephrasing of results of Bridgeland-Toledano Laredo \cite{bt_stab}, Joyce \cite{joy}, Kontsevich-Soibelman \cite{ks}.  
This Frobenius type structure lives in an infinite-dimensional bundle $K \to \Stab(\cC)$ over the space of stability conditions on the category $\cC$ and is given by a collection of holomorphic objects $(\nabla^r, C, \cU, \cV, g)$ with values in $K$, satisfying a set of PDEs. It turns out that the Higgs field $C$ and endomorphism $\cU$ are roughly the same thing as the central charge $Z$ of a stability condition, while the flat connection $\nabla^r$ and endomorphism $\cV$ are roughly the same as the holomorphic generating function $f(Z)$ for counting invariants introduced by Joyce \cite{joy}. In particular we have $\cV(Z) = \ad f(Z)$ for a certain Lie algebra structure on the fibres of $K$. The graded components of the Joyce function are given by the matrix elements $g(x_{\alpha}, \cV(x_{\beta}))$ over a natural basis of sections of $K$, where $g$ is the quadratic form on $K$ given by the Frobenius type structure.

In the important case when $\cC$ is a triangulated category the above construction is always purely formal, even in the simplest examples. The graded components $g(x_{\alpha}, \cV(x_{\beta}))$ of $f(Z)$ are formal infinite sums, and nothing is known about their convergence or even in general how to regard them as formal power series. This is essentially because of the shift functor $[1]\!: \cC \to \cC$ which preserves the class of semistable objects and induces a symmetry of Donaldson-Thomas invariants $\dt(\alpha, Z) = \dt(- \alpha, Z)$ for all classes $\alpha$ in the Grothendieck group $K(\cC)$. This convergence problem for holomorphic generating functions was first discussed in \cite{joy}. Notice that on the contrary when $\cC$ is abelian and sufficiently simple (i.e. of finite length) all sums become finite, and one can even work fully at the motivic level (see \cite{bt_stab}).

\subsection{CV-structure} The Frobenius type structure embeds in a richer structure introduced by Hertling \cite{hert} and called a CV-structure after Cecotti-Vafa. This is suggested naturally by the physical work of Gaiotto, Moore and Neitzke \cite{gmn}. The CV-structure lives on the same bundle $K$ and is given by a collection of non-holomorphic objects $(D, C, \widetilde{C}, \cU, \cQ, \kappa, h)$ with values in $K$, satisfying a set of PDEs. In particular the endomorphism $\cQ$ is a deformation of $\cV$, as we have 
\begin{equation*}
\lim_{\lambda \to 0} \cQ(\lambda Z) = \cV(Z)
\end{equation*}
(see Proposition \ref{DTcvStr}). So in the CV-structure the Joyce function $f(Z)$ is naturally deformed to the operator $\cQ(Z)$. When $\cC$ is a triangulated category the above construction is also purely formal. The matrix elements $g(x_{\alpha}, \cQ(x_{\beta}))$ are ill-defined infinite sums. 
 
\subsection{Formal families} Suppose that $\cC$ is triangulated and admits a heart of a bounded $t$-structure $\cA$ of finite length $n$. Let $U(\cA) \subset \Stab(\cC)$ denote the interior of the set of stability conditions supported on $\cA$. The set $U(\cA)$ is given by stability conditions with heart $\cA$ and central charge $Z$ mapping the effective cone $K_{> 0}(\cA)$ to the open upper half-plane $\mathbb{H}$. On $U(\cA)$ both the Frobenius type and CV-structures can be regarded naturally as formal families of structures defined on a formal neighborhood of $0 \in \C^n$. In particular the ill-defined Joyce function $f(Z)$ and operators $\cV(Z)$, $\cQ(Z)$ become naturally well-defined formal power series $f_{\bf s}(Z)$, $\cV_{\bf s}(Z)$ and $\cQ_{\bf s}(Z)$ in an auxiliary set of parameters ${\bf s} = (s_1, \ldots, s_n)$. This is part of the general results Propositions \ref{DTFrobTypeStrFps}, \ref{DTcvStrFps}. The original ill-defined expressions are recovered for ${\bf s} = (1, \ldots, 1)$, modulo convergence. In this paper we study the convergence problem for the matrix elements $g(x_{\alpha}, \cQ_{\bf s}(Z)(x_{\beta}))$, the CV-deformation of (the graded components of) the Joyce function.

\subsection{Abstract setting and convergence} We will work in an abstract setting modelled on the case of a triangulated category discussed above. This has the advantage of being fully rigorous, independently of the foundational problems of Donaldson-Thomas theory for 3CY triangulated categories, and is achieved by working with abstract continuous families of stability data in the sense of \cite{ks} section 2.3. Thus we fix a lattice $\Gamma$ with a choice of skew-symmetric integral form and an ``effective" strictly convex cone $\Gamma^+ \subset \Gamma$. We state all our results simply in terms of a suitable function $\dt\!: \Gamma \times \Hom^+(\Gamma, \C) \to \Q$ defined on the product of $\Gamma$ with the cone of ``positive" central charges $\Hom^+(\Gamma, \C)$, given by central charges mapping $\Gamma^+$ to the open upper half-plane $\mathbb{H}$. The function $\dt(\alpha, Z)$ should be locally constant in strata of $\Hom^+(\Gamma, \C)$, it should satisfy the wall-crossing formulae of \cite{js, ks} across different strata, and moreover it should enjoy the symmetry $\dt(\alpha, Z) = \dt(-\alpha, Z)$, induced by the shift functor $[1]$ in the categorical case. The conditions on $\dt(\alpha, Z)$ are summarised in the notion of the double of a positive continuous family of stability data parametrised by $\Hom^+(\Gamma, \C)$, see Definition \ref{C0def}. So $\dt(\alpha, Z)$ is modelled on the restriction of the Donaldson-Thomas invariants of a 3CY triangulated category to a domain $U(\cA)$.

Just as in the categorical case such $\dt(\alpha, Z)$ give rise to formal families of Frobenius type and CV-structures, and so to functions $f_{\bf s}(Z)$ and operators $\cV_{\bf s}(Z)$, $\cQ_{\bf s}(Z)$, with $\lim_{\lambda \to 0} \cQ_{\bf s}(\lambda Z) = \cV_{\bf s}(Z)$, see Propositions \ref{DTFrobTypeStrFps} and \ref{DTcvStrFps}. The matrix elements $g(x_{\alpha}, \cQ_{\bf s}(Z) (x_{\beta}))$ are well-defined formal power series and, provided they converge in a neighbourhood of ${\bf s} = (1, \ldots, 1)$, their evaluation $g(x_{\alpha}, \cQ_{(1, \ldots, 1)}( Z) (x_{\beta}))$  is the natural CV-deformation of the graded components of the Joyce functions $f_{(1, \ldots, 1)}(Z)$.
\begin{theorem}\label{mainThm} Fix a central charge $Z_0 \in \Hom^+(\Gamma, \C)$. Suppose that $\dt(\alpha, Z_0)$ grows at most exponentially for $\alpha \in \Gamma$ (in the sense of Definition \ref{expGrowth}). Then for all $\rho > 0$ there exists $\bar{\lambda}$ such that for $\lambda > \bar{\lambda}$ all the formal power series $g(x_{\alpha}, \cQ_{\bf s}(\lambda Z_0) (x_{\beta}))$ converge for $|| \bf s || < \rho$. Let $U \subset \Hom^+(\Gamma, \C)$ denote an open subset such that the exponential growth condition for $\dt(\alpha, Z)$ holds uniformly and all $Z \in U$ are uniformly bounded away from zero on elements of the cone $\Gamma^+$. Then for all sufficiently large $\lambda$ the CV-deformations of the Joyce functions, given by $g(x_{\alpha}, \cQ_{(1, \ldots, 1)}(\lambda Z) (x_{\beta}))$, are well defined and real-analytic on $U$, and uniformly bounded as $\alpha$ varies in $\Gamma$ for fixed $\beta$.  
\end{theorem}
One may expect that in fact we have $|g(x_{\alpha}, \cQ_{(1, \ldots, 1)}(\lambda Z) (x_{\beta}))| \to 0$ as $|| \alpha || \to \infty$ in some fixed norm on $\Gamma\otimes\R$ and for fixed $\beta$, but the methods of the present paper are not sufficient to establish this. Although we have stated our main result in terms of the operators $\cQ_{\bf s}$ it will be clear from the proof that the same statement holds for the full CV-structure. The exponential growth condition for DT type invariants has been investigated quite a lot recently, see e.g. \cite{weist}. It is especially interesting from a physical point of view, see e.g. \cite{cordova} for a recent contribution. Note that a large class of 3CY categories with uniformly bounded DT invariants (to which Theorem \ref{mainThm} applies) is discussed in \cite{bs}.

\subsection{Application to physical field theory} Our methods in this paper are inspired by the fundamental physical work of Gaiotto, Moore and Neitzke \cite{gmn}. We comment on the similarities and differences in Remark \ref{gmnRmk}. Because of this close link our results also say something about certain infinite sums which appear in formal expansions of expectation values of line operators in \cite{gmn}, and we will discuss this in section \ref{gmnSection} (the problem of giving a precise meaning to such expansions was first pointed out explicitly in \cite{neitzke} section 4.2.1). In particular these formal expansions actually give well defined distributions on tori with values in the dual of the charge lattice (see Corollary \ref{gmnCor}).
\subsection{Plan of the paper} Section \ref{formalSec} offers a more detailed introduction to the Frobenius type and CV-structures for Donaldson-Thomas theory in a formal context. Section \ref{stabDataSec} discusses the abstract rigorous approach outlined above. The proof of Theorem \ref{mainThm} is given in section \ref{FunSec} and is based on explicit formulae for Frobenius type and CV-structures in terms of graph integrals (given in section \ref{explicitSec}), uniform estimates on graph integrals (derived in section \ref{estimatesSec}), and a functional equation (studied in section \ref{FunSec}). Section 7 briefly discusses the application to physical field theories. This paper is based on the isomonodromy perspective developed in \cite{bt_stab, bt_stokes} and in \cite{fgs}. We have tried to make the exposition self-contained apart from some proofs from these works which are not reproduced here.
\subsection*{Acknowledgements} We are grateful to Tom Bridgeland, Kwok Wai Chan, Mario Garcia-Fernandez, Kohei Iwaki, Andy Neitzke and Tom Sutherland for helpful discussions related to this paper. Our work greatly benefitted from the workshops ``Current Developments in Mirror Symmetry", TSIMF Sanya, 2014 and ``Geometry from Stability Conditions", Warwick, 2015. We would like to thank the organisers and participants. This research was supported by ERC Starting Grant 307119.
\section{Formal infinite-dimensional picture}\label{formalSec}

This section explains the infinite-dimensional picture of Frobenius type and CV-struc\-ures for Donaldson-Thomas theory. Although some parts of it are purely formal, this section contains essential motivation for our later abstract treatment, and at the same time collects some basic definitions.

We fix a category $\cC$ and assume that there are well-defined numerical Donaldson-Thomas invariants $\dt(\alpha, Z)$ enumerating objects in $\cC$ with class $\alpha \in K(\cC)$ which are semistable with respect to a choice of stability condition $Z$. In particular $\cC$ should be Calabi-Yau and three-dimensional (3CY). We refer to \cite{js, ks} for foundational results. When $\cC$ is triangulated 3CY one should work with stability conditions in the sense of Bridgeland \cite{bridgeland} and assume that there are invariants satisfying the assumptions described in \cite{joy} section 1. In particular the shift functor $[1]\!: \cC \to \cC$ preserves the class of semistable objects and induces a symmetry of Donaldson-Thomas invariants $\dt(\alpha, Z) = \dt(-\alpha, Z)$.  Notice that in this case our notation $Z$ for a stability condition is really a shortcut for the pair $(\cA, Z)$ of a heart of a bounded $t$-structure and a central charge $Z \in \Hom(K(\cA), \C)$.   

\subsection{Frobenius type structure} One can use Donaldson-Thomas theory to attach to $\cC$ a Frobenius type structure on an infinite-dimensional bundle over the space of stability conditions $\Stab(\cC)$. This is essentially a rephrasing of results in \cite{bt_stab, joy, ks}. To explain this fact we start by recalling the definition of a Frobenius type structure on an arbitrary bundle, due to Hertling (\cite{hert} Definition 5.6 (c)).  
\begin{definition}\label{DefFrobType} A \emph{Frobenius type structure} on a holomorphic vector bundle $K \to M$ is a collection of holomorphic objects $(\nabla^r, C, \cU, \cV, g)$, with values in the bundle $K$, where 
\begin{enumerate}
\item[$\bullet$] $\nabla^r$ is a flat connection,
\item[$\bullet$] $C$ is a Higgs field, that is a $1$-form with values in endomorphisms, with $C \wed C = 0$,
\item[$\bullet$] $\cU, \cV$ are endomorphisms,
\item[$\bullet$] $g$ is a symmetric nondegenerate bilinear form,
\end{enumerate}
satisfying the conditions  
\begin{align}\label{FrobTypeCond1}
\nonumber \nabla^r(C) &=0,\\
\nonumber [C, \cU] &= 0,\\
\nonumber \nabla^r(\cV) &= 0,\\
\nabla^r(\cU) - [C, \cV] + C &= 0
\end{align}
plus the conditions on the ``metric" $g$
\begin{align}\label{FrobTypeCond2}
\nonumber \nabla^r(g) &= 0,\\
\nonumber g(C_X a, b) &= g(a, C_X b),\\
\nonumber g(\cU a, b) &= g(a, \cU b),\\
g(\cV a, b) &= -g(a, \cV b).
\end{align}
\end{definition}

Going back to our category $\cC$ we denote by $\bra - , - \ket $ the integral bilinear form on $K(\cC)$ given by the Euler pairing. In the 3CY case this is skew-symmetric. The group-algebra $\C[K(\cC)]$ endowed with the twisted commutative product and Lie bracket induced by $\bra -, - \ket$ becomes a Poisson algebra, known as the Kontsevich-Soibelman algebra. It is generated by monomials $x_{\alpha}, \alpha \in K(\cC)$ with commutative product $x_{\alpha}x_{\beta} = (-1)^{\bra \alpha, \beta\ket} x_{\alpha + \beta}$ and bracket $[x_{\alpha}, x_{\beta}] = (-1)^{\bra \alpha, \beta\ket}\bra \alpha, \beta\ket x_{\alpha + \beta}$. A central charge $Z \in \Hom(K(\cC), \C)$ can be regarded as an endomorphism (in fact a commutative algebra derivation) of $\C[K(\cC)]$ acting by $Z(x_{\alpha}) = Z(\alpha) x_{\alpha}$.

Joyce \cite{joy} introduced a holomorphic generating function for Donaldson-Thomas invariants. It is a formal infinite sum $f(Z)$ of elements of $\C[K(\cC)]$. Morally it defines a holomorphic function on $\Stab(\cC)$ with values in $\prod_{\alpha} \C x_{\alpha}$, encoding the Donaldson-Thomas invariants which enumerate semistable objects in $\cC$. One can reinterpret this construction as giving a Frobenius type structure in the sense of Definition \ref{DefFrobType} on a trivial infinite-dimensional vector bundle $K \to \Stab(\cC)$. 
\begin{definition}\label{Kbundle} The choice of bundle $K \to \Stab(\cC)$ is given by:
\begin{enumerate}
\item[$\bullet$] when $\cC$ is abelian of finite length we denote by $K_{> 0}(\cC)$ the cone of effective classes and let $K$ be the trivial bundle over $\Stab(\cC)$ with fibre $\widehat{\C[K_{> 0}(\cC)]}$, the completion along the ideal generated by the classes of  simple objects $[S_1], \ldots, [S_n]$;
\item[$\bullet$] when $\cC$ is abelian but not of finite length, or when $\cC$ is triangulated we let $K$ be the trivial bundle over $\Stab(\cC)$ with fibre $\prod_{\alpha \in K_{> 0}(\cC)} \C x_{\alpha}$, respectively $\prod_{\alpha \in K(\cC)\setminus\{0\}} \C x_{\alpha}$. In both cases all the constructions below are a priori ill-defined, and we work with formal infinite sums ignoring all convergence questions, just as in \cite{joy} section 5.  
\end{enumerate}
\end{definition}
When summing over elements $\alpha$ of $K(\cC)$ we will always assume $\alpha \neq 0$.
\begin{proposition}\label{DTFrobTypeStr} Let $K \to \Stab(\cC)$ be our trivial infinite-dimensional vector bundle (in particular we have $\bar{\del}_K x_{\alpha} = 0$). Fix a constant $g_0 \in \C^*$. Then there are $(\nabla^r, C, \cU, \cV)$, satisfying the conditions \eqref{FrobTypeCond1}, given by  
\begin{align*}
\nabla^r &= d + \sum_{\alpha} \ad f^{\alpha}(Z) \frac{d Z(\alpha)}{Z(\alpha)},\\
C &= - dZ,\\
\cU &= Z,\\
\cV &= \ad f(Z).
\end{align*}
If moreover $\cC$ is triangulated we can complete these to a Frobenius type structure with the choice
\begin{equation*}
g(x_{\alpha}, x_{\beta}) = g_0 \delta_{\alpha \beta}.
\end{equation*}
Notice that here we use the Lie algebra structure on $\C[K(\cC)]$ just to describe endomorphims of $K$, i.e. we work with a vector bundle not a principal bundle. 
\end{proposition}  
\begin{remark} The function $Z(\alpha)^{-1} f^{\alpha}(Z)$ extends across to locus where $Z(\alpha) = 0$, see \cite{joy} section 5.
\end{remark}
\begin{proof} Let us first clarify our choice of Higgs field. For all $\gamma \in K(\cC)$ the function $Z \mapsto Z(\gamma)$ is a local holomorphic function on $\Stab(\cC)$. So we can define a $1$-form with values in endomorphisms by
\begin{equation*}
d Z (X) x_{\gamma} = ( X Z(\gamma) ) x_{\gamma}
\end{equation*}
for all local holomorphic vector fields $X$. One checks that $d Z \wed dZ = 0$.

To check \eqref{FrobTypeCond1}, \eqref{FrobTypeCond2} one uses repeatedly a PDE on the functions $f^{\alpha}(Z)$ (see \cite{joy} equation (4)), 
\begin{equation}\label{joyPDE}
d f^{\alpha}(Z) = \sum_{\beta, \gamma \in K(\cC)\setminus \{0\},\, \alpha = \beta + \gamma} [f^{\beta}, f^{\gamma}] d\log Z(\gamma).
\end{equation}
Flatness of $\nabla^r$ and covariant constancy of $\cV$ follow from the same computations as in \cite{joy} section 4 (in particular equations (71) - (73)). The other conditions follow from straightforward computations. As an example we have  
\begin{align*}
\nabla^r(d Z) = d^2 Z & + \ad \sum_{\alpha} f^{\alpha}(Z) \frac{d Z(\alpha)}{Z(\alpha)} \wed dZ\\  
&+ d Z \wed \ad \sum_{\alpha} f^{\alpha}(Z) \frac{d Z(\alpha)}{Z(\alpha)}
\end{align*} 
where $\wed$ denotes the composition of endomorphisms combined with the wedge product of forms. Now $d^2 Z =0$, and evaluating on a section $x_{\beta}$ gives a $2$-form with values in $K$
\begin{align*}
\nabla^r(d Z) x_{\beta} &= \sum_{\alpha} [f^{\alpha}(Z), x_{\beta}] (Z(\alpha))^{-1} dZ(\alpha) \wed dZ(\beta)\\
&+ \sum_{\alpha} [f^{\alpha}(Z) , x_{\beta}] (Z(\alpha))^{-1} dZ(\alpha + \beta) \wed dZ(\alpha).
 \end{align*}
But we have $d Z(\alpha + \beta) = d Z(\alpha) + dZ(\beta)$ and the vanishing $\nabla^r(d Z) x_{\beta} = 0$ follows for all $\beta$.

As an example of a condition involving the quadratic form $g$ in the triangulated case we check skew-symmetry of $\cV$. We have
\begin{align*}
g(\cV e_{\alpha}, e_{\beta}) &=  \sum_{\gamma} \tilde{f}^{\gamma}(Z) (-1)^{\bra \gamma, \alpha\ket} \bra \gamma, \alpha \ket g_{\alpha + \gamma, \beta}\\
&= g_0 \sum_{\gamma} \tilde{f}^{\gamma}(Z) (-1)^{\bra \gamma, \alpha\ket}\bra \gamma, \alpha \ket \delta_{\alpha + \gamma, \beta}\\
&= g_0 (-1)^{\bra \beta, \alpha\ket} \bra \beta, \alpha\ket f^{\beta - \alpha}(Z).
\end{align*}
Similarly 
\begin{equation*}
g( e_{\alpha}, \cV e_{\beta}) = g_0 (-1)^{\bra \alpha, \beta\ket} \bra \alpha, \beta\ket f^{\alpha - \beta}(Z). 
\end{equation*}
In the 3CY case we have $f^{\alpha - \beta}(Z) = f^{\beta - \alpha}(Z)$ because of the shift functor.
\end{proof} 
There is a standard construction of a ``first structure" flat connection from a Frobenius type structure. In the Donaldson-Thomas case this has a further scale invariance property.
\begin{lem}\label{confInv} Let $p\!: \Stab(\cC) \times \PP^1_t \to \Stab(\cC)$ denote the projection. Let $\lambda \in \R^+$ denote a scaling parameter. The meromorphic connection on $p^* \Stab(\cC) \times \PP^1_t$ given by
\begin{equation*}
\nabla^r + \frac{C}{t} + \left(\frac{1}{t^2} \cU - \frac{1}{t}\cV \right)dt
\end{equation*}
is flat and invariant under the rescaling $Z \mapsto \lambda Z$, $t \mapsto \lambda t$. In particular the Joyce function $f(Z)$ has the ``conformal invariance" property $f(\lambda Z) = f(Z)$.
\end{lem} 
\begin{proof} Flatness of the connection follows from the conditions \eqref{FrobTypeCond1}. Invariance under the rescaling is equivalent to the property $f(\lambda Z) = f(Z)$ which is established in \cite{joy} section 3.
\end{proof}
\subsection{Convergence problem for triangulated $\cC$} The $K(\cC)$-graded components of $f(Z)$ can be described explicitly. Let $(U \C[K(\cC)], \otimes)$ denote the universal enveloping algebra of $(\C[K(\cC)],  [ - , - ])$. There are explicit formulae for the product $\otimes$, and one has in particular 
\begin{equation*}
x_{\alpha_1} \otimes \cdots \otimes x_{\alpha_k} = c(\alpha_1, \cdots, \alpha_k) x_{\alpha_1 + \cdots + \alpha_k}
\end{equation*}
where $c(\alpha_1, \cdots, \alpha_k) \in \Q$ is given by a sum over connected trees $T$ with vertices labelled by $\{1, \ldots, k\}$, endowed with a compatible orientation, 
\begin{equation}\label{Uproduct}
c(\alpha_1, \cdots, \alpha_k) = \sum_T \frac{1}{2^{k - 1}}\prod_{\{i \to j\} \subset T} (-1)^{\bra \alpha_i, \alpha_j\ket} \bra \alpha_i, \alpha_j\ket.
\end{equation}
Joyce proves that there exist holomorphic functions with branch-cuts $J_n\!: (\C^*)^n \to \C$ such that 
\begin{equation*}
f^{\alpha}(Z) = \sum_{\alpha_1 + \cdots + \alpha_k = \alpha, Z(\alpha_i) \neq 0} J_n(Z(\alpha_1), \ldots , Z(\alpha_k)) \prod_i \dt(\alpha_i, Z) x_{\alpha_1} \otimes \cdots \otimes x_{\alpha_k}
\end{equation*}
and so one has 
\begin{equation*}
f^{\alpha}(Z) = \tilde{f}^{\alpha}(Z) x_{\alpha},
\end{equation*}
where the holomorphic function $\tilde{f}^{\alpha}(Z)$ is given by 
\begin{equation}\label{fFunFormal}
\tilde{f}^{\alpha}(Z) = \sum_{\alpha_1 + \cdots + \alpha_k = \alpha, \, Z(\alpha_i) \neq 0} c(\alpha_1, \ldots, \alpha_k) J_n(Z(\alpha_1), \ldots , Z(\alpha_k)) \prod_i \dt(\alpha_i, Z).
\end{equation}
The crucial point is that the jumps of the functions $J_n(z_1, \ldots, z_n)$ across their branch-cuts can be chosen so as to cancel the jumps of the Donaldson-Thomas invariants $\dt(\alpha_i, Z)$ across walls in $\Hom(K(\cC), \C)$. The functions $J_n(z_1, \ldots, z_n)$ are universal, i.e. they do not depend on the underlying category $\cC$. 

When $\cC$ is triangulated there is a symmetry $\dt(\alpha, Z) = \dt(-\alpha, Z)$ induced by the shift functor $[1]$, so the explicit formula \eqref{fFunFormal} always involves summing over infinitely many decompositions $\alpha_1 + \cdots + \alpha_k = \alpha$ with $c(\alpha_1, \ldots, \alpha_k) \prod_i \dt(\alpha_i, Z) \neq 0$, as soon as $\dt(\alpha_i, Z) \neq 0$ for at least two linearly independent $\alpha_i$. We do not know an example where \eqref{fFunFormal} is known to converge. Indeed the convergence question is a priori ill-posed since no specific summation order has been fixed. The convergence problem for $f(Z)$ seems especially hard because of the conformal invariance property of Lemma \ref{confInv}.

\subsection{CV-structure} The Frobenius type structure of Proposition \ref{DTFrobTypeStr} is part of a more complicated (formal) structure called a CV-structure (after Cecotti and Vafa) in \cite{hert}. This point of view is also suggested naturally by \cite{gmn}. To discuss it we introduce the preliminary notion of a $D C \tilde{C}$-structure, which is also due to Hertling (\cite{hert} Definition 2.9).
\begin{definition} A \emph{$(D C \widetilde{C})$-structure} on a $C^{\infty}$ complex vector bundle $K \to M$ is the collection of $C^{\infty}$ objects $(D, C, \widetilde{C})$ with values in $K$ where
\begin{enumerate}
\item[$\bullet$] $D$ is a connection,
\item[$\bullet$] $C$ is a $(1, 0)$-form with values in endomorphisms of $K$,
\item[$\bullet$] $\widetilde{C}$ is a $(0, 1)$-form with values in endomorphisms of $K$;
\end{enumerate}
satisfying the conditions
\begin{align}\label{DccCond}
\nonumber (D'' + C)^2 &= 0,\quad
\nonumber (D' + \widetilde{C})^2 = 0,\\
\nonumber D'(C) &= 0,\quad
\nonumber D''(\widetilde{C}) = 0,\\
D' D'' + D'' D'  &= -(C\widetilde{C} + \widetilde{C}C)
\end{align}
where $D'$ and $D''$ are the $(1, 0)$ and $(0, 1)$ parts of $D$ respectively.
\end{definition}
\begin{lem}\label{DTdccStr} Let $K \to \Stab(\cC)$ be the vector bundle of Definition \ref{Kbundle}. Then there is a $(D C \widetilde{C})$-structure on $K$ given by
\begin{align*}
D' &= \nabla^r,\quad
D'' = \bar{\del}_K,\\ 
C &= - dZ,\quad
\widetilde{C} = d\bar{Z}.
\end{align*}  
\end{lem}  
\begin{proof} Let $\bar{\del}_K$ denote our fixed (trivial) complex structure on $K$, with $\bar{\del}_K(x_{\alpha}) = 0$. The condition $(D'' + C)^2 = 0$ says that $K$ is holomorphic and $C$ is a holomorphic Higgs bundle on it, which we know already from Proposition \ref{DTFrobTypeStr}. Then $D'(C) = 0$ says that $C$ is flat with respect to $\nabla^r$, which we also know already. The condition $(D' + \widetilde{C})^2 = 0$ says that $\nabla^r$ is flat (known), $(d \bar{Z})^2 = 0$ and $\nabla^r(d \bar{Z}) = 0$ (easily checked). The condition $D''(\widetilde{C}) = 0$ becomes $\bar{\del}_K (d \bar{Z}) = 0$ and can be checked e.g. in local coordinates on $\Stab(\cC)$ given by $z_k = Z(\alpha_k)$ where $\alpha_1, \ldots, \alpha_k$ is a basis for $K(\cC)$. Finally in our case one checks that we have separately $C\widetilde{C} + \widetilde{C}C = 0$ and $D' D'' + D'' D' = 0$. 
\end{proof}

We can now recall the notion of a CV-structure introduced in \cite{hert} Definition 2.16.
\begin{definition} A \emph{CV-structure} on a $C^{\infty}$ complex bundle $K \to M$ is a collection of $C^{\infty}$ objects $(D, C, \widetilde{C}, \kappa, h, \cU, \cQ)$ with values in $K$ where
\begin{enumerate}
\item[$\bullet$] $(D, C, \widetilde{C})$ is a $(D C \widetilde{C})$-structure,
\item[$\bullet$] $\kappa$ is an antilinear involution with $D(\kappa) = 0$ which intertwines $C$ and $\widetilde{C}$, $\kappa C \kappa = \widetilde{C}$,
\item[$\bullet$] $h$ is a hermitian (not necessarily positive) metric, which satisfies $D(h) = 0$, $h(C_X a, b) = h(a, \widetilde{C}_{\bar{X}} b)$ for $(1, 0)$ fields $X$ and which is real-valued on the real subbundle $K_{\R} \subset K$ defined by $\kappa$,
\item[$\bullet$] $\cU$ and $\cQ$ are endomorphisms,
\end{enumerate}
satisfying the conditions
\begin{align}\label{CVstrCond}
\nonumber [C, \cU] &= 0,\\ 
\nonumber D'(\cU) - [C, \cQ] + C &= 0,\\
\nonumber D''(\cU) &= 0,\\
\nonumber D'(\cQ) + [C, \kappa \cU \kappa] &= 0,\\
\nonumber \cQ + \kappa \cQ \kappa &= 0,\\
\nonumber h(\cU a, b) &= h(a, \kappa \cU \kappa b),\\
h(\cQ a, b) &= h(a, \cQ b).
\end{align}
\end{definition}

Let us go back to the case of our bundle $K \to \Stab(\cC)$. Let $\iota$ denote the involution of $K$ acting as complex conjugation, combined with $x_{\alpha} \mapsto x_{-\alpha}$ in the triangulated case. Let $\psi$ be a fixed endomorphism of $K$. Then we can make the following ansatz on part of the data of a CV-structure on $K$:
\begin{enumerate}
\item[$\bullet$] $\kappa$ is the conjugate involution $\Ad_{\psi}(\iota)$,
\item[$\bullet$] the pseudo-hermitian metric $h$ is given by $h(a, b) = g(a, \kappa (b) )$ where $g$ is the quadratic form of Proposition \ref{DTFrobTypeStr}, 
\item[$\bullet$] $\cU$ is the endomorphism $Z$ as in Proposition \ref{DTFrobTypeStr},
\item[$\bullet$] the Higgs field $C$ is given by $- dZ$ as in Proposition \ref{DTFrobTypeStr}, and the anti-Higgs $\widetilde{C}$ is given by $\kappa C \kappa$.    
\end{enumerate}

\begin{proposition}\label{DTcvStr} Let $K \to \Stab(\cC)$ be the vector bundle of Definition \ref{Kbundle}.   
\begin{enumerate}
\item[(a)] There exist endomorphisms $\psi(Z)$, $\cQ(Z)$ and a connection $D$ on $K$ such that the choices of $C$, $\widetilde{C}$, $\kappa$, $h$, $\cU$ above together with $D$ and $\cQ$ give a CV-structure on $K$ (in the abelian case only the conditions not involving $h$ are satisfied). Moreover $\psi$ and $\cQ$ induce fibrewise derivations of $\C[K(\cC)]$ as a commutative algebra.
\item[(b)] Fix $Z$ and let $\lambda \in \R^+$ denote a scaling parameter. Then
\begin{equation*}
\lim_{\lambda \to 0} \cQ(\lambda Z) = \cV, 
\end{equation*}
where $\cV = \ad f(Z)$ is the endomorphism of Proposition \ref{DTFrobTypeStr} (i.e. essentially the Joyce holomorphic generating function).
\end{enumerate} 
\end{proposition}
\begin{proof} We will explain a rigorous approach and prove a rigorous result (which applies to sufficiently simple abelian and triangulated categories) in section \ref{stabDataSec}  and Proposition \ref{DTcvStrFps}. The present formal statement can be ``proved" (in the same sense as Proposition \ref{DTFrobTypeStr}) by the same arguments provided we work with formal infinite sums, ignoring convergence questions.
\end{proof}
There are explicit formulae for the matrix elements $g(x_{\alpha}, \cQ(x_{\beta}))$, $g(x_{\alpha}, \psi(x_{\beta}))$ which are very similar to \eqref{fFunFormal}, see section \ref{explicitSec}. When $\cC$ is triangulated these are always formal infinite sums. They are not known to converge in general and indeed the convergence question is a priori ill-posed since no specific summation order has been fixed.

In the light of Proposition \ref{DTcvStr} (b) it is natural to make the following definition.
\begin{definition} The \emph{CV-deformation} of the Joyce holomorphic generating function $f(Z)$ is the operator $\cQ(Z)$ given by Proposition \ref{DTcvStr} (a).
\end{definition}

There is an analogue of Lemma \ref{confInv}, which gives a new point of view on the conformal invariance property $f(\lambda Z) = f(Z)$. It follows from the proof of Proposition \ref{DTcvStrFps}.
\begin{lem}\label{confLemmaFormal} Let $(D, C, \widetilde{C}, \kappa, h, \cU, \cQ)$ be the CV-structure of Proposition \ref{DTcvStr}. Let $p\!: \Stab(\cC) \times \PP^1_z \to \Stab(\cC)$ denote the projection, and suppose $\lambda \in \R^+$ is a scaling parameter. The meromorphic connection on $p^*K \to \Stab(\cC) \times \PP^1_z$ given by
\begin{equation*}
D + \frac{C}{z} + z \widetilde{C} + \left(\frac{1}{z^2} \cU - \frac{1}{z}\cQ - \kappa \cU \kappa\right)dz
\end{equation*}
is flat. Under the scaling $Z \mapsto \lambda Z$, $z = \lambda t$, $\lambda \to 0$ it flows to the flat connection of Lemma \ref{confInv}.
\end{lem}
\section{Formal families of structures}\label{stabDataSec}
 
Starting with the present section we study the Frobenius type and CV-structures of Donaldson-Thomas theory in a rigorous abstract setting.  

\subsection{Stability data} Fix a finite rank lattice $\Gamma$ with a skew-symmetric bilinear form $\bra - , - \ket$. We denote by $n$ the rank of $\Gamma$.
\begin{definition} We introduce coefficients $c(\alpha_1, \ldots, \alpha_k)$ given by \eqref{Uproduct}. Notice that these only depend on $(\Gamma, \bra - , - \ket)$.
\end{definition}
\begin{definition} A \emph{central charge} $Z$ is a group homomorphism $\Gamma \to \C$.   

A \emph{spectrum} is a function of the form 
\begin{equation*}
(\alpha, Z ) \mapsto \Omega(\alpha, Z ) \in \Q
\end{equation*}
for all $\alpha \in \Gamma$ and $Z$ varying in an open subset $U$ of a linear subspace of $\Hom(\Gamma, \C)$.    

A \emph{distinguished ray}\footnote{The opposite of a ``BPS ray" in physics terminology.} $\ell_{\alpha}(Z) \subset \C^*$ is a ray of the form $\R_{> 0}Z(\alpha)$ such that $\Omega(\alpha, Z ) \neq 0$.

We say that the spectrum $\Omega$ is 
\begin{enumerate}
\item[$\bullet$] \emph{positive} if there exists a $\Z$-basis $\{\gamma_i\}$ of $\Gamma$ such that $\Omega(\alpha, Z)$ vanishes unless $\alpha$ is a nonzero positive integral combination of the $\alpha_i$. In this case we say that $\{\gamma_i\}$ is a positive basis for $\Omega$;
\item[$\bullet$] \emph{symmetric} if 
\begin{equation*}
\Omega(\alpha, Z) = \Omega(- \alpha, Z)
\end{equation*}
for all $\alpha \in \Gamma$, $Z \in U$.
\item[$\bullet$] the \emph{double of a positive spectrum} if $\Omega$ is symmetric and there is a positive spectrum $\widetilde{\Omega}$ such that $\Omega(\alpha, Z) = \widetilde{\Omega}(\pm \alpha, Z)$ for all $\alpha \in \Gamma$, $Z \in U$. 
\end{enumerate}
\end{definition}
\begin{definition} Let $\{\gamma_i\}$ be a fixed basis for $\Gamma$. The locus of \emph{positive central charges} $\Hom^+(\Gamma, \C) \subset \Hom(\Gamma, \C)$ is given by central charges mapping $\{\gamma_i\}$ to the open upper half plane $\mathbb{H} \subset \C$. 
\end{definition}
In the notation of the introduction we have the ``effective cone" $\Gamma^+$ given by nonzero positive linear combinations of the $\{\gamma_i\}$ and $\Hom^+(\Gamma, \C)$ is given by central charges mapping $\Gamma^+$ to $\mathbb{H}$. Recall that in the categorical situation described in the introduction, when $\cA$ is a finite length heart of a 3CY category $\cC$, $\Gamma^+$ is given by $K_{> 0}(\cA)$ and there is a natural positive basis given by the classes of simple objects $[S_1], \ldots, [S_n]$. When well defined the corresponding symmetric spectrum given by Donaldson-Thomas theory is the double of a positive spectrum.
\begin{definition} We say that $Z \in \Hom(\Gamma, \C)$ and $\Omega$ satisfy the \emph{support condition} if there exists a constant $c > 0$ such that picking a norm $|| - ||$ on $\Gamma \otimes \R$ we have
\begin{equation}\label{support}
|Z(\alpha)| > c || \alpha||
\end{equation}
for all $\alpha \in \Gamma$ with $\Omega(\alpha, Z ) \neq 0$. The condition does not depend on the specific choice of norm.
\end{definition}
The support condition was first introduced by Kontsevich-Soibelman in \cite{ks} section 1.2, where its geometric relevance (related to special Lagrangian geometry) was also discussed. Note that if $\Omega$ is positive or the double of a positive spectrum parametrised by $\Hom^+(\Gamma, \C)$ then the support condition is automatically satisfied on $\Hom^+(\Gamma, \C)$. It holds uniformly on all subsets of $\Hom^+(\Gamma, \C)$ where $Z$ is bounded away from zero on the elements of a positive basis $\{\gamma_i\}$. 
\begin{definition}\label{expGrowth} We say that a spectrum $\Omega$ \emph{grows at most exponentially at $Z $} if  there is a $\lambda > 0$ such that 
\begin{equation}\label{expBound}
\sum_{\alpha \in \Gamma}   |\Omega(\alpha, Z )| \exp( -  |Z(\alpha)| \lambda ) < \infty.
\end{equation} 
\end{definition}

The spectra coming from Donaldson-Thomas theory share a crucial property: they are \emph{continuous}, that is they define \emph{continuous families of stability data} (in the sense of \cite{ks} section 2.3) on the Kontsevich-Soibelman graded Lie algebra modelled on the underlying lattice $(\Gamma, \bra - , - \ket)$. We now give the details of this property.
\begin{definition}\label{ksAlgebra} The \emph{Kontsevich-Soibelman Poisson algebra} $\g_{\Gamma}$ is the (associative, commutative) group algebra $\C[\Gamma]$ endowed with the twisted multiplication and Lie bracket induced by $\bra - , - \ket$: $\g_{\Gamma}$ is generated by $x_{\alpha}$, $\alpha \in \Gamma$, with commutative product $x_{\alpha} x_{\beta} = (-1)^{\bra \alpha, \beta\ket}x_{\alpha + \beta}$ and bracket $[x_{\alpha}, x_{\beta}] = (-1)^{\bra \alpha, \beta \ket}\bra\alpha, \beta\ket x_{\alpha + \beta}$. To avoid confusion we write $\exp_*$ for the commutative algebra exponential in $\g_{\Gamma}$. 
 \end{definition}
One checks that $\g_{\Gamma}$ is indeed Poisson, i.e. inner Lie algebra derivations are commutative algebra derivations. A central charge $Z$ defines an endomorphism of $\g_{\Gamma}$ by $Z(x_{\alpha}) = Z(\alpha)x_{\alpha}$. This is in fact a commutative algebra derivation. 

\begin{definition} Fix a basis $\{\gamma_i\}$ as above. We write $\g_{> 0} \subset \g_{\Gamma}$ for the monoid generated by $x_{\alpha}$ where $\alpha$ is nonzero and has nonnegative coefficients with respect to the basis. We let $\widehat{\g}_{> 0}$ be the completion of $\g_{> 0}$ along the ideal $(x_{\gamma_1}, \ldots, x_{\gamma_n})$. 
\end{definition}
For a choice of a strictly convex cone $V \subset \C^*$ and a central charge $Z$ satisfying the support condition, there is a complete Poisson Lie algebra $\widehat{\g}_{\Gamma, V, Z}$ topologically generated by elements $x_{\alpha}$ with $Z(\alpha) \in V$. The algebra $\widehat{\g}_{\Gamma, V, Z}$ does not change when varying $Z$ as long as no ray $\R_{> 0}Z(\alpha)$ crosses the boundary $\del V$, so with this choice of $Z$ we will denote the complete algebra simply by $\widehat{\g}_{\Gamma, V}$. We denote by $\exp(\widehat{\g}_{\Gamma, V})$ the corresponding formal Lie group (i.e. the group law is defined formally by the usual Baker-Campbell-Hausdorff formula).  

Let $\dt(\alpha, Z )$ denote the M\"obius transform of $\Omega$, 
\begin{equation}\label{mobius}
\dt(\alpha, Z ) = \sum_{k > 0, k | \alpha} \frac{1}{k^2} \Omega(k^{-1}\alpha, Z ). 
\end{equation}
\begin{definition}\label{C0def} The \emph{family of stability data on $\g_{\Gamma}$ parametrised by $U$} corresponding to the spectrum $\Omega$ is the $\g_{\Gamma}$-valued function given by
\begin{equation*}
(\alpha, Z ) \mapsto \dt(\alpha, Z ) x_{\alpha}.
\end{equation*}
This family of stability data on $\g_{\Gamma}$ is continuous in the sense of \cite{ks} section 2.3 if all $Z \in U$ satisfy the support condition, and for every fixed strictly convex cone $V \subset \C^*$ the group element
\begin{equation}\label{contCond}
\prod^{\to, Z}_{\ell \subset V} \exp\left( \sum_{Z(\alpha) \in \ell} \dt(\alpha, Z ) x_{\alpha}\right) \in \exp(\widehat{\g}_{\Gamma, V})
\end{equation}
is constant as long as no distinguished ray $\ell_{\alpha}(Z)$ crosses the boundary $\del V$, where $\prod^{\to, Z}$ denotes the operator writing the ensuing formal Lie group elements from left to right according to the clockwise $Z$-order.

We say that the spectrum $\Omega$ is continuous is the corresponding family of stability data on $\g_{\Gamma}$ is. We say that the family of stability data $\dt(\alpha, Z)$ is positive, symmetric, or the double of a positive family if the corresponding condition is satisfied by the underlying spectrum $\Omega(\alpha, Z)$ given by inverting \eqref{mobius},
\begin{equation*}
\Omega(\alpha, Z) = \sum_{k | \alpha} \frac{1}{k^2} m(k) \dt(k^{-1} \alpha, Z)
\end{equation*} 
where $m$ denotes the M\"obius function.
\end{definition}

It will be important for us to regard the group element in \eqref{contCond}, under suitable conditions, as a product of explicit ``symplectomorphisms". 
\begin{definition}\label{genericCharge} A central charge $Z \in \Hom(\Gamma, \C)$ is \emph{generic} if elements $x_{\alpha}, x_{\beta}$ with $Z(\alpha), Z(\beta)$ lying on the same ray $\ell$ have vanishing Lie bracket (i.e. $\bra \alpha, \beta \ket = 0$). We say that $Z$ is \emph{strongly generic} if $Z(\alpha), Z(\beta)$ lying on the same ray $\ell$ implies that $\alpha, \beta$ are linearly dependent. We write $\Hom^{sg}(\Gamma, \C)$ for the locus of strongly generic central charges.
\end{definition}
\noindent For $\Omega \in \Q$, $Z(\beta) \in V$ let $T^{\Omega}_{\beta}$ denote the element of $\Aut(\widehat{\g}_{\Gamma, V})$ given by
\begin{equation*}
T^{\Omega}_{\beta}(x_{\alpha}) = x_{\alpha} (1 - x_{\beta})^{ \bra \beta, \alpha \ket \Omega}
\end{equation*} 
(the right hand side denoting a formal power series expansion). In fact $T^{\Omega}_{\beta}$ is a Poisson automorphism (it preserves the Lie bracket). This follows from the identity
\begin{equation*}
T^{\Omega}_{\beta} = \exp_{D(\widehat{\g}_{\Gamma, V})}\left(- \Omega \sum_{k \geq 1} \frac{[x_{k \beta}, - ]}{k^2}\right).
\end{equation*}
Let $\Aut(\widehat{\g}_{\Gamma, V})$ and $D(\widehat{\g}_{\Gamma, V})$ denote the group of automorphisms of $\widehat{\g}_{\Gamma, V}$ as a commutative, associative algebra, respectively the $\widehat{\g}_{\Gamma, V}$-module of commutative algebra derivations. Kontsevich-Soibelman \cite{ks} section 2.5 noticed that for generic $Z$  there is a factorisation in $\Aut(\widehat{\g}_{\Gamma, V})$
\begin{equation}\label{PoissonOps}
\exp_{D(\widehat{\g}_{\Gamma, V})}\left(\sum_{Z(\alpha) \in \ell} - \dt(\alpha, Z) [ x_{\alpha}, -] \right) = \prod_{Z(\beta) \in \ell}T^{\Omega(\beta, Z)}_{\beta}.  
\end{equation}
The continuity condition becomes the constraint that the product of Poisson automorphisms $\prod^{\to, Z}_{\ell \subset V} \prod_{Z(\alpha) \in \ell} T^{\Omega(\alpha, Z)}_{\alpha}$ remains constant in the locus of generic central charges (even when crossing the nongeneric locus) as long as no rays supporting a nonvanishing factor enter or leave $V$.

The notion of a continuous family of stability data with values in $\g$ makes sense quite generally for a $\Gamma$-graded Lie algebra $\g$ over $\Q$ (see \cite{ks} Section 1.2). It will be important for us to consider continuous families with values in the Lie algebra $\g_{\Gamma}\pow{s_1, \ldots, s_n}$ endowed with the Poisson Lie bracket extended from $\g_{\Gamma}$ by $\C\pow{s_1, \ldots, s_n}$-linearity. 

Let $\{\gamma_i\}$ be a basis for $\Gamma$ and introduce formal parameters ${\bf s} = (s_1, \ldots, s_n)$. We decompose $\alpha \in \Gamma$ as $\alpha = \sum_i a_i \gamma_i$ and set $[\alpha]_{\pm} = \sum_{i} [a_i]_{\pm} \gamma_i$ where $[a_i]_{\pm}$ denote the positive and negative parts. For all $\alpha \in \Gamma$ we write ${\bf s}^{\alpha}$ for the Laurent monomial $\prod_i s^{a_i}_i$. In particular ${\bf s}^{[\alpha]_+ - [\alpha]_-}$ is a monomial (not just a Laurent monomial). Let $J \subset \g_{\Gamma}\pow{\bf s}$ denote the ideal generated by $s_1, \ldots, s_n$.
\begin{definition} With a fixed choice of basis as above we write $T^{\Omega}_{\beta, {\bf s}}$ for the element of $\Aut(\g_{\Gamma}\pow{\bf s})$ given by  
\begin{equation*}
T^{\Omega}_{\beta, {\bf s}}(x_{\alpha}) = x_{\alpha} (1 - {\bf s}^{[\beta]_+ - [\beta]_-} x_{\beta})^{ \bra \beta, \alpha \ket \Omega}
\end{equation*} 
\end{definition} 
\begin{lem}\label{isomono} Let $\Omega$ be the double of a positive spectrum. Fix a positive basis $\{\gamma_i\}$. Suppose that $\Omega$ is continuous, parametrised by $\Hom^+(\Gamma, \C)$. Then the family of automorphisms $T^{\Omega}_{\beta, {\bf s}} \in \Aut(\g_{\Gamma}\pow{\bf s})$ comes from a continuous family of stability data with values in $\g_{\Gamma}\pow{\bf s}$ via the construction in \eqref{PoissonOps}. In particular the products $\prod^{\to, Z}_{\ell \subset V} \prod_{Z(\alpha) \in \ell} T^{\Omega(\alpha, Z)}_{\alpha, {\bf s}}$ for all fixed strictly convex sectors $V$ remain constant in the locus of generic central charges in $\Hom^+(\Gamma, \C)$ (even when crossing the nongeneric locus) as long as no rays supporting a nonvanishing factor enter or leave $V$.
\end{lem}
\begin{proof} Suppose that $\Omega$ is the double of a positive, continuous spectrum parametrised by $\Hom^+(\Gamma, \C)$. Then the continuity condition given by constancy of the formal Lie group element \eqref{contCond} holds if and only if it holds for all strictly convex cones $V$ contained in the open upper half-plane $\mathbb{H}$. On such a cone $V \subset \mathbb{H}$ the constancy condition for \eqref{contCond} is compatible with the extra grading by ${\bf s}$ by the Baker-Campbell-Hausdorff formula. 
\end{proof}
\begin{remark} The idea of working with such formal families is natural from the point of view of scattering diagrams described e.g. in \cite{gps}.
\end{remark}
\begin{definition} Let $\Omega$ be a positive, continuous spectrum parametrised by $\Hom^{+}(\Gamma, \C)$ and fix a positive basis. The corresponding Joyce function $f(Z)$ is the $\widehat{\g}_{> 0}$-valued function with graded components $\widetilde{f}^{\alpha}(Z) x_{\alpha}$ given by the expression  \eqref{fFunFormal}. This is well-defined because there are only finitely many possible decompositions in \eqref{fFunFormal} for each fixed $\alpha \in \Gamma_{> 0}$.     
\end{definition}
\begin{definition} Let $\Omega$ be the double of a positive, continuous spectrum parametrised by $\Hom^{+}(\Gamma, \C)$. The corresponding Joyce function $f_{\bf s}(Z)$ is the function with values in $\g_{\Gamma}\pow{\bf s}$ with $\Gamma$-graded components $\widetilde{f}^{\alpha}_{\bf s}(Z) x_{\alpha}$, where 
\begin{align}\label{fFunFps}
\nonumber \tilde{f}^{\alpha}_{\bf s}(Z) = &\sum_{\alpha_1 + \cdots + \alpha_k = \alpha,\, Z(\alpha_i) \neq 0} c(\alpha_1, \ldots, \alpha_k) J(Z(\alpha_1), \ldots , Z(\alpha_k))\\ &\prod_i {\bf s}^{[\alpha_i]_+ - [\alpha_i]_-} \dt(\alpha_i, Z).
\end{align}
This is well-defined because there are only finitely many decompositions in \eqref{fFunFps} modulo $J^N$ for $N \gg 1$. 
\end{definition}
\subsection{Formal families of Frobenius type and CV-structures} 
Fix a basis $\{\gamma_i\}$ of $\Gamma$. We consider a spectrum $\Omega$ which is either positive with respect to $\{\gamma_i\}$ or the double of such a positive spectrum.
\begin{definition}\label{KbundleFormal}  We introduce a holomorphic bundle $K \to \Hom^+(\Gamma, \C)$ given by:
\begin{enumerate}
\item[$\bullet$] if $\Omega$ is positive, $K$ is the trivial bundle with fibre $\widehat{\g}_{> 0}$;
\item[$\bullet$] if $\Omega$ is the double of a positive spectrum, $K$ is the trivial bundle with fibre $\g_{\Gamma}\pow{\bf s}$. 
\end{enumerate}
\end{definition}
Although our main result Theorem \ref{mainThm} only concerns the double of a positive spectrum, for the sake of completeness we summarise the results in the case of a positive spectrum in the following Proposition. The part concerning the Frobenius type structure follows from the results of \cite{bt_stokes}, while the claims about the CV-structure are proved exactly as in Proposition \ref{DTFrobTypeStrFps} below, working with $\widehat{\g}_{> 0}$ rather than $\g_{\Gamma}\pow{\bf s}$.   
\begin{proposition} Let $\Omega$ be a positive, continuous spectrum parametrised by $\Hom^+(\Gamma, \C)$. Let $K \to \Hom^+(\Gamma, \C)$ be the vector bundle of Definition \ref{KbundleFormal}. Then the obvious analogues of Propositions \ref{DTFrobTypeStr}, \ref{DTcvStr} and Lemmas \ref{confInv}, \ref{confLemmaFormal} hold.
\end{proposition}
Turning to the double of a positive spectrum, the construction of a formal family of Frobenius type structures follows from the results of \cite{bt_stokes}, so we only give a sketch of the proof. 
\begin{proposition}\label{DTFrobTypeStrFps} Let $\Omega$ be the double of a positive, continuous spectrum parametrised by $\Hom^{+}(\Gamma, \C)$. Let $K \to \Hom^+(\Gamma, \C)$ be the vector bundle of Definition \ref{KbundleFormal}, with fibre $\g_{\Gamma}\pow{\bf s}$. Then there is a $\C\pow{{\bf s}}$-linear Frobenius type structure on $K$ with flat holomorphic connection given by 
\begin{equation*}
\nabla^r_{{\bf s}} = d + \sum_{\alpha} \ad f^{\alpha}_{\bf s}( Z) \frac{d Z(\alpha)}{Z(\alpha)},
\end{equation*}
with residue endomorphism
\begin{equation*}
\cV_{{\bf s}} = \ad f_{\bf s}( Z)
\end{equation*}
and with $C, \cU, g$ given by $- dZ, Z$ and the quadratic form of Proposition \ref{DTFrobTypeStr}, extended by $\C\pow{{\bf s}}$-linearity. In other words the equations $(\nabla^r_{\bf s})^2 = 0$ and \eqref{FrobTypeCond1} - \eqref{FrobTypeCond2} hold as identities of formal power series in the formal parameters $s_1, \ldots, s_n$.

In particular the coefficients of the formal power series \eqref{fFunFps} in ${\bf s}$ are well-defined holomorphic functions on $\Hom^+(\Gamma, \C)$.
\end{proposition}
\begin{proof}[Proof (sketch)] It follows from Lemma \ref{isomono} and the results of \cite{bt_stokes} (on the explicit inverse of a certain Stokes map) that the functions $f^{\alpha}_{\bf s}(Z)$ satisfy the PDE \eqref{joyPDE} as formal power series in ${\bf s}$. Then the corresponding Frobenius type structure is constructed as in the proof of Proposition \ref{DTFrobTypeStr}.
\end{proof}
Let $\iota$ denote the involution of $K$ acting as complex conjugation combined with $x_{\alpha} \mapsto x_{-\alpha}$. Note that $\iota$ is an anti-linear commutative algebra automorphism. Let $\psi_{\bf s}$ be a fixed invertible endomorphism of $K$. Then we can make the following ansatz on part of the data of a $\C\pow{\bf s}$-linear CV-structure on $K$:
\begin{enumerate}
\item[$\bullet$] $\kappa_{\bf s}$ is the conjugate involution $\Ad_{\psi_{\bf s}}(\iota)$,
\item[$\bullet$] the pseudo-hermitian metric $h_{\bf s}$ is given by $h(a, b) = g(a, \kappa_{\bf s} (b) )$ where $g$ is the quadratic form in Proposition \ref{DTFrobTypeStr},
\item[$\bullet$] $\cU$ is the endomorphism $Z$ extended by $\C\pow{\bf s}$-linearity, 
\item[$\bullet$] the Higgs field $C$ is given by $- dZ$ extended by $\C\pow{\bf s}$-linearity, and the anti-Higgs field $\widetilde{C}_{\bf s}$ by $\kappa_{\bf s} C \kappa_{\bf s}$.    
\end{enumerate}
\begin{proposition}\label{DTcvStrFps} Suppose we are in the situation of Proposition \ref{DTFrobTypeStrFps}.       

\begin{enumerate}
\item[(a)] There exist $\C{\pow{\bf s} }$-linear endomorphisms $\psi_{\bf s}$ and $\cQ_{\bf s}$ and a connection $D_{\bf s}$ on $K$ such that the choices of $C$, $\widetilde{C}_{\bf s}$, $\kappa_{\bf s}$, $h_{\bf s}$, $\cU_{\bf s}$ above together with $\cQ_{\bf s}$ give a $\C{\pow{\bf s} }$-linear CV-structure on $K$. In other words the equations \ref{DccCond} and \ref{CVstrCond} hold as identities of formal power series in ${\bf s}$. Moreover $\psi_{\bf s}$ and $\cQ_{\bf s}$ induce fibrewise $\C\pow{\bf s}$-linear derivations of $\g_{\Gamma}\pow{\bf s}$ as a commutative algebra.
\item[(b)] We have
\begin{equation*}
\lim_{\lambda \to 0} \cQ_{\bf s}(\lambda Z) = \cV_{\bf s}, 
\end{equation*}
where $\cV_{\bf s} = \ad f_{\bf s}( Z)$ is the endomorphism of Proposition \ref{DTFrobTypeStrFps} (i.e. essentially the formal family of Joyce holomorphic generating functions given by \eqref{fFunFps}).
\end{enumerate}
\end{proposition}
\begin{proof} We consider the family of automorphisms of the commutative algebra $\g_{\Gamma}\pow{\bf s}$ induced by $T^{\Omega(\alpha, Z)}_{\alpha, {\bf s}}$ for a fixed $Z \in \Hom^{sg}(\Gamma, \C) \cap \Hom^+(\Gamma, \C)$. In \cite{fgs} section 3 the corresponding Riemann-Hilbert factorization problem for a map $X: \C^* \to \Aut(\g_{\Gamma}\pow{\bf s})$ is studied. This is the problem of finding a function $X(z)$ with values in $\Aut(\g_{\Gamma}\pow{\bf s})$ such that, for all $N \geq 1$ and $\alpha \in \Gamma$, the class of $X(z)(x_{\alpha})$ in $\g_{\Gamma}\pow{\bf s}/J^N$ is a holomorphic function of $z$ in the complement of the distinguished rays $\ell$ with $\ell \neq \ell_{\pm \alpha(Z)}$, and for $z_0 \in \ell$ we have
\begin{equation*}
X(z^+_0)(x_{\alpha}) = X(z^-_0) \circ \prod_{Z(\beta) \in \ell} T^{\Omega(\beta, Z)}_{\beta, {\bf s}}(x_{\alpha}) \mod J^N
\end{equation*}   
where $z^{\pm}_0$ denote the limits in the counterclockwise, respectively clockwise directions. Note that by working modulo $J^N$ there are only finitely many branch-cuts. In \cite{fgs} Lemma 3.10 a distinguished explicit solution $X(z)$ is constructed, satisfying some additional properties (this construction is very much inspired by ideas in \cite{gmn}). We denote this distinguished family of solutions as $Z$ varies in $\Hom^{sg}(\Gamma, \C) \cap \Hom^+(\Gamma, \C)$ by $X(z, Z)$, and also set\footnote{In \cite{fgs} $\tilde{X}(z, Z)$ is denoted by $Y(z, Z)$. In the present paper we have reserved the latter symbol for a flat section of the connection given by \eqref{isoConnections} below in order to simplify the notation, see Proposition \ref{DefFlatSection}.} 
\begin{equation*}
\tilde{X}(z, Z) = X(z, Z) \circ \exp_{D(\g_{\Gamma}\pow{\bf s})}(- z^{-1} Z - z \bar{Z}). 
\end{equation*}
Consider the flat connection on $\Hom^+(\Gamma, \C) \times \PP^1_z$ given by 
\begin{equation*}
\nabla^{tr} = d - \frac{dZ}{z} + z d\bar{Z} + \left( \frac{1}{z^2} Z - \bar{Z} \right) dz.
\end{equation*}
We may regard $\nabla^{tr}$ as a flat connection on the trivial vector bundle with fibre $\g_{\Gamma}\pow{\bf s}/J^N$. Together with $g(a, \iota(b))$ it defines a CV-structure on the trivial vector bundle with fibre $\g_{\Gamma}\pow{\bf s}/J^N$ on $\Hom^+(\Gamma, \C)$. We pull back $\nabla^{tr}$ locally on a sector $\Sigma$ between consecutive branch-cuts by $\tilde{X}(z, Z) \mod J^N$ to the locally defined flat connection 
\begin{equation*}
\nabla^{str}|_{\Sigma} = d - \frac{1}{z} \tilde{X} \cdot dZ + z \tilde{X} \cdot d\bar{Z} + d_Z \tilde{X} \circ \tilde{X}^{-1} + \left( \frac{1}{z^2} \tilde{X} \cdot Z - \tilde{X} \cdot \bar{Z}  + \del_z \tilde{X} \circ \tilde{X}^{-1} \right) dz.  
\end{equation*}  
By \cite{fgs} sections 3.7 and 3.9, $\nabla^{str}$ glues over different sectors $\Sigma$ and is induced by the class $\mod J^N$ of a well-defined real-analytic flat connection on $\g_{\Gamma}\pow{\bf s} \to \Hom^+(\Gamma, \C) \times \PP^1_z$ of the form
\begin{equation*}
\nabla^{str}(Z) = d + \mathcal{B}^{(0)}(Z) + \frac{1}{z} \mathcal{B}^{(-1)}(Z) + z \mathcal{B}^{(1)}(Z) + \left(\frac{1}{z^2}\mathcal{A}^{(-1)}(Z) + \frac{1}{z}\mathcal{A}^{(0)}(Z) + \mathcal{A}^{(1)}(Z) \right) dz. 
\end{equation*}
Moreover the $\mathcal{A}^{(i)}$, $\mathcal{B}^{(i)}$ are derivations (respectively $1$-forms with values in derivations) of $\g_{\Gamma}\pow{\bf s}$ and we have 
\begin{align}\label{conjugate}
\nonumber \mathcal{A}^{(1)}(Z) &= -\iota \mathcal{A}^{(-1)}(Z) \iota, \quad
\nonumber \mathcal{A}^{(0)}(Z) = - \iota \mathcal{A}^{(0)}(Z) \iota,\\
\mathcal{B}^{(1)}(Z) &= \iota \mathcal{B}^{(-1)}(Z) \iota,\quad\quad
\mathcal{B}^{(0)}(Z) = \iota \mathcal{B}^{(0)}(Z) \iota.
\end{align} 
By \cite{fgs} section 3.7 the limit $\tilde{X}_0(Z) = \lim_{z \to 0} \tilde{X}(z, Z)$ is well-defined, and we have
\begin{align}\label{pullback}
\nonumber \tilde{X}^{-1}_0 \cdot \nabla^{str}(Z) &= d + \Ad_{\tilde{X}^{-1}_0}\mathcal{B}^{(0)}(Z) - \frac{1}{z} d Z + z \Ad_{\tilde{X}^{-1}_0} \mathcal{B}^{(1)}(Z)\\
& + \left( \frac{1}{z^2}Z + \frac{1}{z} \Ad_{\tilde{X}^{-1}_0} \mathcal{A}^{(0)} + \Ad_{\tilde{X}^{-1}_0} \mathcal{A}^{(1)} \right) dz. 
\end{align}
Notice that by \eqref{conjugate} and \eqref{pullback} we have 
\begin{align*}
\Ad_{\tilde{X}^{-1}_0} \mathcal{A}^{(1)} &= - \Ad_{\tilde{X}^{-1}_0} \Ad_{\iota} \mathcal{A}^{(-1)}\\
&= - \Ad_{\tilde{X}^{-1}_0} \Ad_{\iota} \Ad_{\tilde{X}_0}( Z ),\\
\Ad_{\tilde{X}^{-1}_0} \mathcal{B}^{(1)} &= \Ad_{\tilde{X}^{-1}_0} \Ad_{\iota} \mathcal{B}^{(-1)}\\
&= \Ad_{\tilde{X}^{-1}_0} \Ad_{\iota} \Ad_{\tilde{X}_0}( - d Z ), 
\end{align*}
so using the conjugate involution $\kappa = \Ad_{\tilde{X}^{-1}_0}(\iota)$ we may rewrite \eqref{pullback} as
\begin{align*} 
\tilde{X}^{-1}_0 \cdot \nabla^{str}(Z) &= d + \Ad_{\tilde{X}^{-1}_0}\mathcal{B}^{(0)}(Z) - \frac{1}{z} d Z + z \kappa (- dZ) \kappa \\
& + \left( \frac{1}{z^2}Z + \frac{1}{z} \Ad_{\tilde{X}^{-1}_0} \mathcal{A}^{(0)} - \kappa Z \kappa \right) dz. 
\end{align*}
Then the flat connection $\tilde{X}^{-1}_0 \cdot \nabla^{str}(Z)$  together with $\kappa$ define the required $\g_{\Gamma}\pow{\bf s}$-linear CV-structure, with $D = d + \Ad_{\tilde{X}^{-1}_0}\mathcal{B}^{(0)}(Z)$, $C = - dZ$, $\tilde{C} = \kappa (- dZ)\kappa$, $\cU = Z$, $\cQ = - \Ad_{\tilde{X}^{-1}_0} \mathcal{A}^{(0)}$, $h(a, b) = g(a, \kappa b)$. The automorphism in the statement of the Proposition is given by $\psi_{\bf s}(Z) = \tilde{X}^{-1}_0(Z)$.

The limit 
\begin{equation*}
\lim_{\lambda \to 0} \cQ(\lambda Z) = \cV(Z)
\end{equation*}  
is proved in \cite{fgs} Theorem 4.2. We provide a sketch of the argument. Since they are constructed from a solution to the Riemann-Hilbert factorization problem, the family of connections on $\PP^1_z$
\begin{equation}\label{isoConnections} 
d + \left(\frac{1}{z^2} Z - \frac{1}{z}\cQ_{\bf s}(\lambda Z) - \lambda^2 \kappa_{\bf s}(\lambda Z)\,Z\kappa_{\bf s}(\lambda Z)\right)dz
\end{equation}
parametrised by $\Hom^+(\Gamma, \C)$ are isomonodromic, with constant generalized monodromy at  $z = 0$ for generic $Z$ given by rays $\ell$ with Stokes factors $\prod_{Z(\beta) \in \ell} T^{\Omega(\beta, Z)}_{\beta, {\bf s}}(x_{\alpha})$. One checks that the limit as $\lambda \to 0$ is well-defined and equals
\begin{equation*} 
d + \left(\frac{1}{z^2} Z - \frac{1}{z}\lim_{\lambda\to 0}\cQ_{\bf s}(\lambda Z)\right)dz.
\end{equation*}
The result follows from a uniqueness result proved in \cite{bt_stokes}.
\end{proof}
\begin{corollary} The statement of Lemma \ref{confLemmaFormal} holds for the Frobenius type and CV-structures constructed in Propositions \ref{DTFrobTypeStrFps} and \ref{DTcvStrFps}.  
\end{corollary}
\begin{definition} We write $\nabla_{\bf s}(Z, \lambda)$ for the family of meromorphic connections on $\PP^1$ given by \eqref{isoConnections}. 
\end{definition}
\section{Explicit formulae}\label{explicitSec}

In this section we give an explicit formula for the operator $\cQ_{\bf s}(Z)$. We always assume that we fix a continuous symmetric spectrum $\Omega$ parametrized by $\Hom^+(\Gamma, \C)$ which is the double of a positive spectrum. We also assume that a positive basis $\{\gamma_i\}$ has been fixed.
\subsection{Explicit formula for flat sections}\label{explicitSection1} In the rest of the paper we write $T$ for a finite \emph{rooted} tree, with vertices decorated by elements of $\Gamma$. We assume that $T$ is connected unless we state explicitly otherwise. Denote the root decoration by $\alpha_T$. The operation of removing the root produces a finite number of new connected, $\Gamma$-decorated trees $T \mapsto \{T_j\}$. We introduce holomorphic functions with branch-cuts
\begin{equation*}
H_T\!: \C^* \times \Hom^+(\Gamma, C) \cap \Hom^{sg}(\Gamma, \C) \times \R_{\geq 0} \to \C^*
\end{equation*}
attached to trees by the recursion  
\begin{equation}\label{HFun}
H_{T}(z, Z, \lambda) = \frac{1}{2\pi i} \int_{\ell_{\alpha_{T}}}\frac{d w}{w} \frac{z }{w - z} \exp( - Z(\alpha_{\T}) w^{-1} -   \lambda^2 \bar{Z}(\alpha_{T}) w) \prod_j H_{T_j}(w),
\end{equation}
with the initial condition $H_{\emptyset} = 1$. We also introduce weights $W_T(Z) \in \Gamma \otimes \Q$ attached to trees by 
\begin{equation}\label{WFun}
W_T(Z) = \frac{1}{|\Aut(T)|} \dt(\alpha_T, Z) \alpha_T \prod_{\{v \to w\} \subset T}  \bra \alpha(v), \alpha(w)\ket \dt(\alpha(w), Z).   
\end{equation}
We can pair $W_T(Z)$ with $\beta \in \Gamma$ to obtain $ \bra \beta, W_T(Z)\ket \in \Q$. We extend this pairing to possibly \emph{disconnected} trees $T$ with finitely many connected components $T_i$ by setting
\begin{equation*}
 \bra \beta, W_T(Z)\ket = \prod_i  \bra \beta, W_{T_i}(Z) \ket.
\end{equation*}  
\begin{definition} A \emph{distinguished sector} $\Sigma$ is the inverse system under inclusion of sectors $\Sigma_N$ between consecutive distinguished rays $\ell$ such that 
\begin{equation*}
\sum_{Z(\alpha) \in \ell} \dt(\alpha, Z ) {\bf s}^{[\alpha]_+ - [\alpha]_-} x_{\alpha} \notin J^N. 
\end{equation*}
This is well defined because for each $N$ there are only finitely many distinguished rays for which the above sum does not vanish modulo $J^N$.  
\end{definition}
\begin{proposition}\label{DefFlatSection}  The automorphism $Y_{\bf s}(z, Z, \lambda)$ of $\g_{\Gamma}\pow{\bf s}$ acting by
\begin{align}\label{explicitY}
\nonumber Y_{\bf s}(z, Z, \lambda)(x_{\beta}) &= x_{\beta} \exp_* \sum_{T} \bra  \beta, W_T(Z) \ket  H_T(z, Z, \lambda) \prod_{v \in T} {\bf s}^{[\alpha(v)]_+ - [\alpha(v)]_-} x_{\alpha(v)}\\
&= x_{\beta} \sum_{\operatorname{disconnected } T} \bra \beta, W_T(Z) \ket H_T(z, Z, \lambda) \prod_{v \in T} {\bf s}^{[\alpha(v)]_+ - [\alpha(v)]_-} x_{\alpha(v)}
\end{align}
induces a flat section of $\nabla_{\bf s}(Z, \lambda)$ on each distinguished sector $\Sigma$.
\end{proposition}
\begin{proof} This is proved in \cite{fgs} section 4 (see in particular section 4.3). Note that in the notation of the proof of Proposition \ref{DTcvStrFps} we have $Y_{\bf s}(z, Z, \lambda)= \tilde{X}^{-1}_0(\lambda Z) \circ X(\lambda z, \lambda Z)$.
\end{proof}
\subsection{Explicit formula for coefficients} We proceed to discuss explicit formulae for the \emph{coefficients} of $\nabla_{\bf s}(z, Z, \lambda)$ rather that its flat sections. Let $A_{\bf s} \in D(\g_{\Gamma}\pow{\bf s})$ denote the opposite of the connection $1$-form of $\nabla_{\bf s}(z, Z, \lambda)$, so
\begin{equation*}
\del_z Y_{\bf s}(z, Z, \lambda) = A_{\bf s}\,Y_{\bf s}(z, Z, \lambda)
\end{equation*}
(where the right hand side is given by the composition of linear maps). Locally $A_{\bf s}$ is given by the composition of linear maps $(\del_z Y_{\bf s}) Y^{-1}_{\bf s} $, where 
\begin{align*}
\del_z Y_{\bf s} (x_{\alpha}) &= \del_z (Y_{\bf s}(z, Z, \lambda)(x_{\alpha}))\\
&= Y_{\bf s}(z, Z, \lambda)(x_{\alpha}) \sum_{T}  \bra \alpha,   W_T(Z) \ket \del_z H_T(z, Z, \lambda) \prod_{v \in T} {\bf s}^{[\alpha(v)]_+ - [\alpha(v)]_-} x_{\alpha(v)}. 
\end{align*}
Notice that a map of the form $(\del_z Y) Y^{-1}$ where $Y$ takes values in automorphisms of a commutative algebra is automatically a derivation.

Because of its specific form $Y_{\bf s}$ can be inverted explicitly via multivariate Lagrange inversion. Recall that this gives a concrete way to invert self-maps of a ring of formal power series $R\pow{\xi_1, \ldots, \xi_m}$ of the form $\xi_i \mapsto \xi_i \exp( - \Phi_i(\xi_1, \ldots, \xi_m))$ for some $\Phi_i(\xi_1, \ldots, \xi_m) \in R\pow{\xi_1, \ldots, \xi_m}$, where $R$ is a ground $\C$-algebra. 

To reduce the problem of explicitly inverting $Y_{\bf s}$ to a multivariate Lagrange inversion we notice that since $Y_{\bf s}$ is a commutative algebra automorphism it is enough to calculate $Y^{-1}_{\bf s}(x_{\gamma_i})$ for $i = 1, \ldots, n$. We may then try to apply a Lagrange inversion formula over the base ring $R = \C\pow{\bf s}$. A further technical difficulty arises since $Y_{\bf s}$ is a self-map of a ring of Laurent polynomials $\C\pow{\bf s}[x^{\pm 1}_{\gamma_1}, \ldots, x^{\pm 1}_{\gamma_n}]$ over $\C\pow{\bf s}$ rather than formal power series. To remedy this we introduce $2n$ auxiliary parameters ${\bf \xi} = (\xi_1, \ldots, \xi_{2n})$ and set for $\alpha \in \Gamma$
\begin{equation*}
{\bf \xi}^{\alpha} = \prod^n_{i =1} \xi^{[\alpha_i]_+}_i \prod^{2n}_{j = n+1} \xi^{- [\alpha_j]_-}_j.
\end{equation*}
Consider the auxiliary problem of inverting the self-map of $\C\pow{\bf s}\pow{\bf{\xi}}$ given by
\begin{equation*}
(\xi_1, \ldots, \xi_{2n}) \mapsto (F_1(\xi), \ldots, F_{2n}(\xi)), \quad F_i(\xi) = \xi_i \exp( - \Phi_i(\xi))
\end{equation*}
where we choose 
\begin{equation*}
\Phi_i(\xi) =  - \sum_{T}  \bra \gamma_i, W_T(Z) \ket  H_T(z, Z, \lambda) \prod_{v \in T} {\bf s}^{[\alpha(v)]_+ - [\alpha(v)]_-} \xi^{\alpha(v)}
\end{equation*}
for $i = 1, \ldots, n$, respectively 
\begin{equation*}
\Phi_i(\xi) =  \sum_{T}  \bra \gamma_i, W_T(Z) \ket  H_T(z, Z, \lambda) \prod_{v \in T} {\bf s}^{[\alpha(v)]_+ - [\alpha(v)]_-} \xi^{\alpha(v)}
\end{equation*}
for $i = n+1, \ldots, 2n$. If we can solve this then specialising $\xi_i = x_{\gamma_i}$ for $i = 1, \ldots, n$, respectively $\xi_i =  x^{-1}_{\gamma_i}$ for $i = n+1, \ldots, 2n$ determines the inverse $Y^{-1}_{\bf s}$ completely. Going back to the auxiliary problem, suppose that we can solve the equations
\begin{equation}\label{LagrangeEqu}
E_i({\bf \xi}) = \xi_i \exp\left(\Phi_i(E_1({\bf \xi}), \ldots, E_{2m}({\bf \xi}))\right).
\end{equation} 
Then we have
\begin{equation*}
F_i(E_1 , \ldots, E_{2m} ) = E_i  \exp\left(- \Phi_i(E_1 , \ldots, E_{2m} )\right) = \xi_i,
\end{equation*}
so the inverse is given by $(\xi_1, \ldots, \xi_{2m}) \mapsto (E_1({\bf \xi} ), \ldots, E_{2m}({\bf \xi}))$.
\begin{lem} There exist unique $E_i({\bf \xi}) \in \C[\bf \xi]\pow{\bf s}$ solving \eqref{LagrangeEqu}. Moreover for each multi-index ${\bf k} \in \Z^{2m}_{> 0}$ the coefficient of ${\bf \xi}^{\bf k}$ in $E_i({\bf \xi})$ is given by 
\begin{equation}\label{GoodEqu}
[{\bf \xi}^{\bf k}] E_i({\bf \xi} ) = [{\bf \xi}^{\bf k}] \det(\delta_{p q} + \xi_p \del_{q} \Phi_p({\bf \xi} )) \xi_i \exp\left(- \sum_j k_j \Phi_j({\bf \xi} )\right).
\end{equation}
\end{lem}
\begin{proof} Regard $\Phi_i({\bf \xi})$ as formal power series in $\xi_1, \ldots, \xi_{2m}$ with coefficients in $\C\pow{\bf s}$. Applying the multivariate Lagrange inversion formula in a version due to Good (see e.g. \cite{gessel} Theorem 3, equation (4.5)) over the ground ring $\C\pow{\bf s}$ shows that there exists a unique solution $(E_1, \ldots, E_{2m})$ of \eqref{LagrangeEqu} where $E_i \in \C\pow{\bf s}\pow{\bf \xi}$ are given by \eqref{GoodEqu}. That we have in fact $E_i({\bf \xi}) \in \C[\bf \xi]\pow{\bf s}$ follows from the definition of $\Phi_i({\bf \xi})$.
\end{proof}
For a multi-index ${\bf k} \in \Z^{2m}_{> 0}$, ${\bf k} = (k_1, \ldots, k_{2m})$ we set $[ {\bf k} ] = \sum^{m}_{i = 1} (k_i - k_{m + i}) \gamma_i \in \Gamma$. Note that we have $\prod^m_{i = 1} x^{k_i}_{\gamma_i} \prod^{m}_{j = 1} x^{-k_{j + m}}_{\gamma_j} = \pm x_{[{\bf k}]}$ for a unique choice of sign, depending only on ${\bf k}$. We denote this sign by $(-1)^{\bf k}$.
\begin{corollary} For $i = 1, \ldots, m$ and $\alpha \in \Gamma$ we have
\begin{align*}
g(x_{\alpha}, Y^{-1}_{\bf s}(x_{\gamma_i})) &= g_0 \sum_{[ {\bf k} ] = \alpha} (-1)^{\bf k} [{\bf \xi}^{\bf k}] E_i({\bf \xi} )\\ 
&= g_0 \sum_{[ {\bf k} ] = \alpha} (-1)^{\bf k} [{\bf \xi}^{\bf k}] \det(\delta_{p q} + \xi_p \del_{q} \Phi_p({\bf \xi} )) \xi_i \exp\left(- \sum_j k_j \Phi_j({\bf \xi} )\right) \in \C\pow{\bf s}.  
\end{align*}
\end{corollary}
\begin{corollary} For $i = 1, \ldots, m$ we have
\begin{align}\label{explicitA}
\nonumber A_{\bf s}(z, Z, \lambda)(x_{\gamma_i}) = &\sum_{\alpha \in \Gamma} \sum_{[ {\bf k} ] = \alpha} (-1)^{\bf k} [{\bf \xi}^{\bf k}] \det(\delta_{p q} + \xi_p \del_{q} \Phi_p({\bf \xi} )) \xi_i \exp\left(- \sum_j k_j \Phi_j({\bf \xi} )\right) \\& Y(x_{\alpha}) \sum_{T} \bra \alpha, W_T(Z)\ket \del_z H_T(z, Z, \lambda) \prod_{v \in T} s^{[\alpha(v)]_+ - [\alpha(v)]_-} x_{\alpha(v)} \in \g_{\Gamma}\pow{\bf s}.
\end{align}
In particular the CV-deformation $\cQ_{\bf s}(\lambda Z)$ is the derivation of $\g_{\Gamma}\pow{\bf s}$ determined by
\begin{equation*}
\cQ_{\bf s}(\lambda Z) (x_{\gamma_i}) = \operatorname{Res}_{z = 0} A_{\bf s}(x_{\gamma_i}). 
\end{equation*}
\end{corollary}

\section{Estimates on graph integrals}\label{estimatesSec}

In this section we study the graph integrals $H_T(z, Z, \lambda)$. We fix a tree $T$ and $z^*\in \C^*$ which does not belong to any of the rays $\ell_{\alpha(v)}$ for $v \in T$. We will write
\begin{equation*}
H_T(Z, \lambda) = H(z^*, Z, \lambda).
\end{equation*} 
\begin{proposition}\label{mainEstimateProp}
Let $T$ be a $\Gamma$-labelled rooted tree with $n$ vertices. Then there exist universal constants $\bar{\lambda}, C_1, C_2 > 0$, depending only on the constant in the support condition \eqref{support} (in particular, independent of $n, z^*$), such that   
\begin{equation}\label{mainEstimate}
|H_T(Z, \lambda)| \leq  C^n_1  \exp( - C_2 \sum_{v \in T} |Z(\alpha(v))| \lambda)
\end{equation} 
for all $\lambda > \bar{\lambda}$.
\end{proposition}
\noindent The crucial point is that the estimate \eqref{mainEstimate} holds up to the boundary of $\Hom^{sg}(\Gamma, \C)$ where some distinguished rays collide, and irrespective of the presence of accumulation points for the set of distinguished rays for a fixed central charge $Z$.

We now collect some necessary preliminaries to the proof of Proposition \ref{mainEstimateProp}. For nonzero $\alpha \in \Gamma$, $\lambda > 0$ we introduce a function
\begin{equation*}
u_{\alpha, \lambda}(s) = \frac{1}{s} \exp( - \lambda |Z(\alpha)|( s^{-1} +  s)) \chi_{(0, + \infty)}.
\end{equation*}
Notice that $u_{\alpha, \lambda} \in C^{\infty}(\R) \cap L^p(\R)$ for all $1 \leq p \leq \infty$. 
\begin{definition} We denote by $\mathcal{H}$ the \emph{Hilbert transform} on the real line, a bounded linear operator mapping $L^p(\R)$ to itself for $1 < p < \infty$ (by a theorem of M. Riesz, see e.g. \cite{horm} section 3.2). In particular we have by definition
\begin{equation*}
\mathcal{H}[u_{\alpha, \lambda}](s) = \pv \int^{\infty}_0 \frac{d w}{w} \frac{1}{s - w} \exp[- \lambda |Z(\alpha)|( w^{-1} +  w)]. 
\end{equation*}
\end{definition}

By the Riesz theorem $\mathcal{H}[u_{\alpha, \lambda}](s)$ lies in $L^p(\R)$ for $1 < p < \infty$. Standard regularity results imply that $\mathcal{H}[u_{\alpha, \lambda}](s)$ is in $C^1(\R_{s} \times \R_{\lambda > 0})$ and that we can differentiate under the $\mathcal{H}$ operator. One can check by explicit computation that $\mathcal{H}[u_{\alpha, \lambda}]$ as well as $\del_s \mathcal{H}[u_{\alpha, \lambda}]$ lie in $L^{\infty}(\R_s \times \R_{\lambda > 0})$.

We consider a class of functions defined iteratively by
\begin{equation}\label{iterativeFun}
\tau s^{l} u_{\alpha, \lambda}( s) \prod^{k}_{i = 1} \mathcal{H}[v_i](s) 
\end{equation}
where $\tau \in \C^*$, $l = 0, 1$ and each $v_i$ is again of the form \eqref{iterativeFun} for some $\alpha_i \in \Gamma$. Examples include $u_{\alpha_0, \lambda} \prod^k_{i = 1}\mathcal{H}[u_{\alpha_i, \lambda}]$ as well as $u_{\alpha_0, \lambda} \mathcal{H}[u_{\alpha_1, \lambda}\mathcal{H}[u_{\alpha_2, \lambda} \cdots ]]$.
\begin{lem}\label{iterativeLem} Let $u$ be a function of the form \eqref{iterativeFun}, with $m$ corresponding lattice elements $\alpha_1, \ldots, \alpha_m$ (not necessarily distinct). Then there are constants $C_1, C_2, \bar{\lambda}_1$, independent of $m$, depending only on $\tau$ and a common lower bound on $|Z(\alpha_1)|, \ldots, |Z(\alpha_m)|$, such that for all $\lambda > \bar{\lambda}_1$ we have
\begin{equation*}
|| u(s) ||_{L^1} \leq C^m_1 \prod^m_{i = 1}\exp(- C_2 |Z(\alpha_i)|\lambda).
\end{equation*}
\end{lem}
\begin{proof} We will argue by induction on $m$. Using the specific form \eqref{iterativeFun} of $u$ we find
\begin{align*}
|| u(s) ||_{L^1} \leq \prod^{k}_{i = 1} || \mathcal{H}[v_i]( s) ||_{\infty} || \tau s^l u_{\alpha, \lambda}( s)||_{L^1} 
\end{align*}
provided all the $\mathcal{H}[v_i]$ are bounded. By explicit computation (for example using the Laplace approximation for exponential integrals) the factor $||\tau s^l u_{\alpha, \lambda}(\tau s)||_{L^1}$   
has the required uniform exponential decay dominated by $C_1 \exp( - C_2 |Z(\alpha)|\lambda)$ for some fixed uniform $C_2$ and all sufficiently large $C_1$. So we focus on $|| \mathcal{H}[v_i](s)||_{\infty}$. By an  elementary  Sobolev embedding we have 
\begin{equation*}
|| \mathcal{H}[v_i](s)||_{\infty} \leq c_1 || \mathcal{H}[v_i] ||_{W^{1, 2}}
\end{equation*} 
so we start by controlling the $L^2$ norms $|| \mathcal{H}[v_i] ||_{L^2}$, $||\del_s \mathcal{H}[v_i] ||_{L^2}$. By $L^2$ boundedness of $\mathcal{H}$ and the fact that it commutes with $\del_s$ we find
\begin{equation*}
|| \mathcal{H}[v_i] ||_{L^2} \leq c_2 || v_i ||_{L^2}, \quad || \del_s \mathcal{H}[v_i] ||_{L^2} \leq c_2 || \del_s v_i ||_{L^2},
\end{equation*}
that is
\begin{equation*}
|| \mathcal{H}[v_i] ||_{\infty} \leq c_1 c_2 || v_i ||_{W^{1,2}}.
\end{equation*}
We have reduced the problem to finding exponential bounds on $|| v_i ||_{L^2}$ and $|| \del_s v_i ||_{L^2}$. Writing 
\begin{equation*} 
v_i = \tau_i s^{l_i} u_{\beta, \lambda}(s) \prod^{k_i}_{j = 1} \mathcal{H}[w_j]( s) 
\end{equation*}
we get 
\begin{align*}
|| v_i ||_{L^2} & \leq   \prod^{k_i}_{j = 1} || \mathcal{H}[w_j]( s) ||_{\infty} || \tau_i s^{l_i} u_{\beta, \lambda}( s) ||_{L^2}, \\
|| \del_s v_i ||_{L^2} & \leq  \sum^{k_i}_{r = 1} || \mathcal{H}[\del_s w_r]( s) ||_{L^2} \prod_{j \neq r} || \mathcal{H}[w_j](  s) ||_{\infty} || \tau_i s^{l_i} u_{\beta, \lambda}(s) ||_{\infty}\\
& + \prod^{k_i}_{j = 1} || \mathcal{H}[w_j]( s) ||_{\infty} || \del_s ( \tau_i s^{l_i} u_{\beta, \lambda}(s)) ||_{L^2}\\
& \leq c_3 \left(\sum^{k_i}_{r = 1} || \del_s w_r (s) ||_{L^2} \prod_{j \neq r} || \mathcal{H}[w_j]( s) ||_{\infty} || \tau_i s^{l_i} u_{\beta, \lambda}(s) ||_{\infty} \right. \\
&\left. + \prod^{k_i}_{j = 1} || \mathcal{H}[w_j]( s) ||_{\infty} || \del_s (\tau_i s^{l_i} u_{\beta, \lambda}(s)) ||_{L^2}\right).
\end{align*}
Notice that we chose the $L^2$ norm for the factor $\mathcal{H}[\del_s w_r](s)$ rather than the supremum norm so that no further derivatives are required to control this. By explicit computation (e.g. Laplace approximation) the factors $ || \tau_i s^{l_i} u_{\beta, \lambda}(s) ||_{L^2}$, $|| \tau_i s^{l_i} u_{\beta, \lambda}(s) ||_{\infty}$ and $|| \del_s ( \tau_i s^{l_i} u_{\beta, \lambda}(s)) ||_{L^2}$ are all dominated by $C_1 \exp(- C_2|Z(\beta)| \lambda)$ for some fixed uniform $C_2$ and all large $C_1$. Assuming inductively that we have the required exponential bounds on the norms $|| w_j ||_{L^2}$, $|| \del_s w_j||_{L^2}$ for all $j = 1, \ldots, k_i$ the inequalities above imply a bound (denoting by $m_i$ the number of lattice elements $\alpha^i_j$ attached to $v_i$, counted with their multiplicities)
\begin{equation*}
|| v_i ||_{W^{1,2}} \leq c^{m_i}_4 \prod^{m_i}_{j = 1} \exp(- C_2 |Z(\alpha^i_j)|\lambda). 
\end{equation*}
Taking the product over $i = 1, \ldots, m$ yields the result, with $C_1 = c_4$.
\end{proof}
\begin{proof}[Proof of Proposition \ref{mainEstimateProp}] In the course of the proof we use the notation $s_v$ for $v \in T$ to denote positive real integration variables. Hopefully these will not be confused with the parameters ${\bf s}$ of our formal families; the latter never appear in the present section. Parametrising the ray $\ell_{\alpha(v)}$ for $v \in T$ by  
\begin{equation*} 
\lambda^{-1} (|Z(\alpha(v))|)^{-1} Z(\alpha(v)) s_v, \, s_v \in \R_{> 0}
\end{equation*} 
for each $v \in T$ turns $H_T(Z, \lambda)$ into an iterated integral along the positive real line $(0, +\infty)$. Pick a vertex $w \in T$ with unique incoming vertex $v$ distinct from the root. There is a corresponding factor in $H_T(Z, \lambda)$ given by 
\begin{equation*}
(2\pi i)^{-1} \int^{\infty}_0  ds_w  \frac{\tau_w s_v}{\tau_w s_v - s_w} u_{\alpha(w), \lambda}(s_w),  
\end{equation*}
with
\begin{equation*}
\tau_w = \frac{|Z(\alpha(w))|}{Z(\alpha(w))} \frac{Z(\alpha(v))}{|Z(\alpha(v))|}. 
\end{equation*}
Let $c_1, \delta > 0$ denote positive constants to be determined independently of $T$ (in particular, independently of $n$). Suppose that there is an edge $\{v \to w\} \subset T$ such that $|\Im(\tau_w)| < \delta$. Choose the edge for which $\Im(\tau_w)$ is the smallest possible in $T$ (that is, such that the sine of the convex positive angle between the corresponding rays $\ell_{\alpha(v)}$, $\ell_{\alpha(w)}$ is less than $\delta$, and the smallest among edges in $T$). Notice that by our minimal choice of $v \to w$ there are no further rays $\ell_{\alpha(w')}$ with $w \to w'$ between $\ell_{\alpha(v)}$ and $\ell_{\alpha(w)}$. We claim that for sufficiently small $\delta$ there is a uniform $c_1$ such that
\begin{equation*}
|H_T(Z, \lambda)| \leq c_1( |H_{T, 1}(Z, \lambda)| + |H_{T, 2}(Z, \lambda)|),
\end{equation*}
where the iterated integrals $H_{T, 1}(Z, \lambda)$ and $H_{T, 2}(Z, \lambda)$ are obtained by replacing the factor
\begin{align}\label{PlemeljInt}
(2\pi i)^{-2} \int^{\infty}_0 ds_v \frac{\tau_v s_o}{\tau_v s_o - s_v} u_{\alpha(v), \lambda}(s_v) \int^{\infty}_0  ds_w \frac{\tau_w s_v}{\tau_w s_v - s_w} u_{\alpha(w), \lambda}(s_w) 
\end{align}
attached to the subgraph $\{o \to  v \to w\} \subset T$ (denoting by $o$ the unique vertex mapping to $v$) by the Hilbert transform 
\begin{equation}\label{pv}
(2\pi i)^{-2}  \int^{\infty}_0  ds_v  \frac{\tau_v s_o}{\tau_v s_o - s_v} u_{\alpha(v), \lambda}(s_v)  s_v \mathcal{H}[u_{\alpha(w), \lambda}](  s_v)
\end{equation}
in the case of $H_{T, 1}(Z, \lambda)$, respectively by 
\begin{equation}\label{residue}
(2\pi i)^{-1} \int^{\infty}_0  ds_v  \frac{\tau_v s_o}{\tau_v s_o - s_v} u_{\alpha(v), \lambda}(s_v) u_{\alpha(w), \lambda}(s_v)
\end{equation}
in the case of $H_{T, 2}(Z, \lambda)$. This holds because by the classical Sokhotski-Plemelj theorem in complex analysis (see e.g. \cite{horm} section 3.2) the limit of the factor \eqref{PlemeljInt} as $\tau_w \to 1$ is given by the sum of the principal value part \eqref{pv}, and the residue part \eqref{residue}, with suitable signs (determined by whether $\Im(\tau_w) \to 0$ from below or above). The $\tau_w \to 1$ limit holds uniformly for all $\alpha(v)$, $\alpha(w)$, so the claim follows. 

Notice that we can estimate the residue part \eqref{residue} by
\begin{equation*}
||u_{\alpha(w), \lambda}||_{\infty}\left| (2\pi i)^{-1} \int^{\infty}_0  ds_v  \frac{\tau_v s_o}{\tau_v s_o - s_v} u_{\alpha(v), \lambda}(s_v) \right|.
\end{equation*} 
Let $T_2$ be the rooted, $\Gamma$-labelled tree obtained from $T$ by contracting the edge $\{v \to w\}\subset T$ to a single vertex decorated by $\alpha(v)$. By the estimate above we have 
\begin{equation*}
|H_{T, 2}(Z, \lambda)| \leq ||u_{\alpha(w), \lambda}||_{\infty} |H_{T_2}(Z, \lambda)|), 
\end{equation*}
so 
\begin{equation}\label{residueEstimate}
|H_T(Z, \lambda)| \leq c_2( |H_{T, 1}(Z, \lambda)| + ||u_{\alpha(w), \lambda}||_{\infty} |H_{T, 2}(Z, \lambda)|).
\end{equation}

On the other hand edges $\{v \to w\}\subset T$ for which we have a fixed lower bound $|\Im(\tau_w)| \geq \delta > 0$ can be ``integrated out": let $T_3 \subset T$ be the (rooted, $\Gamma$-labelled) subtree obtained by chopping out the (rooted, $\Gamma$-labelled) subtree $T_4 \subset T$ with root $w$. Then there is a constant $c_3$, depending only on $\delta$, such that
\begin{equation*}
|H_T(Z, \lambda)| \leq c_3  |H_{T_3}(Z, \lambda)|  |H_{T, 4}(Z, \lambda)| ,
\end{equation*}
wherer $H_{T, 4}(Z, \lambda) $ equals essentially $H_{T_4}(Z, \lambda) $, but with root factor in the integral replaced with  
\begin{equation*}
\int^{\infty}_0  d s_w u_{\alpha(w), \lambda}(s_w). 
\end{equation*}

We can now proceed inductively applying the two steps described above, decreasing the number of vertices of $T$ or increasing the number of $\mathcal{H}$ operators inserted. The process stops in a finite number of steps, yielding residual functions $H_i(Z, \lambda)$ for a finite set of indices $i \in I$, with cardinality $| I | \leq 2^n$,  such that
\begin{equation*}
|H_T(Z, \lambda)| \leq c^n_4 \left( \sum_{i \in I} |H_i(Z, \lambda)|\right)
\end{equation*}
where $c_4 > 0$ does not depend on $T$. By construction each $| H_i(Z, \lambda)| $ is bounded by a finite product of factors of the form $|| u_{\alpha(w), \lambda}||_{\infty}$ or $|| u(s) ||_{L^1}$, where $u$ belongs to the class of functions \eqref{iterativeFun}. So by Lemma \ref{iterativeLem} and repeated application of \eqref{residueEstimate}  each $|H_i(Z, \lambda)|$ is bounded by $C^n_1 \exp( - C_2 \sum_{v \in T} |Z(\alpha(v))| \lambda)$ for absolute constants $C_1, C_2$ and all $\lambda > \bar{\lambda}$ (independently of $T$).
The bound \eqref{mainEstimate} now follows with that same $C_2, \bar{\lambda}$ and taking the constant $C_1$ in the statement to be $2  C_1 c_4 $ in our present notation. 
\end{proof}
\section{Functional equation and convergence}\label{FunSec}

In this section we complete the proof of our main result Theorem \ref{mainThm}. We fix a continuous symmetric spectrum $\Omega$ parametrised by $\Hom^+(\Gamma, \C)$ which is the double of a positive spectrum. 
\begin{definition} Fix constants $c_1, c_2, \lambda > 0$ and a collection of formal power series $S_{\alpha}({\bf s}) \in \C\pow{\bf s}$ for $\alpha \in \Gamma$. Define a new collection $\mathcal{F}[S]_{\beta}(\bf s) \in \C\pow{\bf s}$ for $\beta \in \Gamma$ by
\begin{equation*}
\mathcal{F}[S]_{\beta}({\bf s}) = \prod_{\alpha \in \Gamma} (1 - c_1 \exp( - c_2 |Z(\alpha)|\lambda) {\bf s}^{[\alpha]_+ - [\alpha]_-}  S_{\alpha}({\bf s}))^{|\bra \beta, \alpha \ket| |\Omega(\alpha, Z)|}. 
\end{equation*} 
\end{definition}
Let us write $S^{(0)}$ for the family of constant formal power series 
\begin{equation*}
S^0_{\beta}({\bf s}) = 1 \in \C\pow{\bf s}. 
\end{equation*}
for all $\beta \in \Gamma$. We define inductively for $i \geq 0$
\begin{equation*}
S^{(i + 1)}_{\beta}({\bf s}) = \mathcal{F}[S^{(i)}]_{\beta}({\bf s}).
\end{equation*} 
\begin{lem} Fix $\bar{\rho} > 0$. There exists $\bar{\lambda} > 0$, depending only on $\bar{\rho}$ and the constants in the support and exponential growth conditions \eqref{support}, \eqref{expBound}, such that for $\lambda \geq \bar{\lambda}$ all the formal power series $ S^{(i)}_{\beta}({\bf s})$ converge for $|| {\bf s} || < \bar{\rho}$, uniformly for $i \geq 0$.  
\end{lem}
\begin{proof} We argue by induction on $i$. For $r > 0$ we write $B_{r} = \{ {\bf s} \in \C^n : || {\bf s} || < r \}$ for the open ball.  Pick a norm $|| - ||$ on $\Gamma \otimes \C$. Suppose that $\bar{\rho} > 0$, $\bar{\lambda} > 0$ and $c_3 > 0$ are constants such that $ S^{(i)}_{\alpha}({\bf s})$ converges absolutely and uniformly in compact subsets of $B_{\bar{\rho}}$ and moreover we have 
\begin{equation}\label{SiBound} 
|S^{(i)}_{\alpha}({\bf s})| < c_3 e^{|| \alpha ||}.
\end{equation} 
for all ${\bf s} \in B_{\bar{\rho}}$, $\lambda > \bar{\lambda}$, $\alpha \in \Gamma$. In the case of $S_0$ we can choose the constants $\bar{\rho}, \lambda > 0$ arbitrarily, while $c_3$ is a positive constant that only depends on the choice of norm $|| - ||$.  

The infinite product 
\begin{equation*}
\prod_{\alpha \in \Gamma} (1 - c_1 \exp( - c_2 |Z(\alpha)|\lambda) {\bf s}^{[\alpha]_+ - [\alpha]_-} S^{({i})}_{\alpha}({\bf s}))^{|\bra \beta, \alpha \ket| |\Omega(\alpha, Z)|}
\end{equation*}
converges absolutely and uniformly in compact subsets of $B_{\bar{\rho}}$ if and only if this happens for the series 
\begin{equation}\label{logSeries}
\sum_{\alpha \in \Gamma} |\bra \beta, \alpha \ket| |\Omega(\alpha, Z)| \log(1 - c_1 \exp( - c_2 |Z(\alpha)|\lambda) {\bf s}^{[\alpha]_+ - [\alpha]_-} S^{(i)}_{\alpha}({\bf s})).
\end{equation} 
There is a uniform constant $c_4 > 0$ such that for all sufficiently large $\lambda$, depending only on the constant in the support condition \eqref{support} and the inductive bound \eqref{SiBound}, the series \eqref{logSeries} is bounded by 
\begin{equation}\label{logSeries2}
c_4 ||\beta||\sum_{\alpha \in \Gamma} ||\alpha|| |\Omega(\alpha, Z)| c_1 \exp( - c_2 |Z(\alpha)|\lambda) c_3 \bar{\rho}^{[\alpha]_+ - [\alpha]_-} e^{||\alpha||}.  
\end{equation}
This bound is independent of $i$. If the spectrum $\Omega(\alpha, Z)$ has at most exponential growth then the series \eqref{logSeries2} converges for all sufficiently large $\lambda$, depending only on $\bar{\rho}$, the support condition \eqref{support} and the exponential growth condition \eqref{expBound}. Moreover for all sufficiently large $\lambda$, depending only on \eqref{support}, \eqref{expBound}, the sum of the series  is bounded by $|| \beta||\log c_3$, from which we get
\begin{equation*} 
|S^{(i+1)}_{\beta}({\bf s})| < c_3 e^{|| \beta ||}
\end{equation*} 
in $B_{\bar{\rho}}$. So if we choose our initial $\bar{\lambda}$ sufficiently large, depending only on $\bar{\rho}$ and the conditions \eqref{support}, \eqref{expBound}, the induction goes through.
\end{proof}
Let $T$ denote a $\Gamma$-labelled rooted tree as usual. We write $\operatorname{depth}(T)$ for the length of the longest oriented path in $T$. Let us denote by $\mu|\Omega|(\alpha, Z)$ the M\"obius transform of the function $|\Omega(\alpha, Z)|$, 
\begin{equation*}
\mu|\Omega|(\alpha, Z) = \sum_{k > 0,\, k | \alpha} \frac{1}{k^2}|\Omega(k^{-1}\alpha, Z)|.
\end{equation*} 
Note that in general $\mu|\Omega|(\alpha, Z) x_{\alpha}$ is not a continuous family of stability data in $\g_{\Gamma}$, and $|\Omega(\alpha, Z)|$ is not a continuous spectrum. This is completely irrelevant for our purposes, since we will only use the obvious bound
\begin{equation*}
|\dt(\alpha, Z)| \leq \mu|\Omega|(\alpha, Z). 
\end{equation*}
Let us introduce weights $\widetilde{W}_T(Z) \in \Gamma\otimes \Q$ by
\begin{equation*} 
\widetilde{W}_T(Z) = \frac{1}{|\Aut(T)|} \mu|\Omega|(\alpha_T, Z) \alpha_T \prod_{\{v \to w\} \subset T} \bra \alpha(v),  \alpha(w)\ket \mu|\Omega|(\alpha(w), Z).   
\end{equation*}
\begin{lem} We have 
\begin{equation*} 
S^{(i)}_{\beta}({\bf s}) = \sum_{\operatorname{disconnected} T,\, \operatorname{depth}(T) \leq i} c^{|T|}_1 |\bra \beta , \widetilde{W}_T(Z) \ket| \exp(- c_2 \sum_{v \in T} |Z(\alpha(v))| \lambda)\prod_{v \in T} {\bf s}^{[\alpha(v)]_+ - [\alpha(v)]_-}.  
\end{equation*}
\end{lem}
\begin{proof} We write
\begin{equation*}
S^{(i + 1)}_{\beta} = \exp \sum_{\alpha \in \Gamma} |\bra \beta, \alpha \ket| |\Omega(\alpha, Z)| \log(1 - c_1 \exp( - c_2 |Z(\alpha)|\lambda) {\bf s}^{[\alpha]_+ - [\alpha]_-}  S^{(i)}_{\alpha}({\bf s})).
\end{equation*}
The result follows from expanding $\log(1 - c_1 \exp( - c_2 |Z(\alpha)|\lambda) {\bf s}^{[\alpha]_+ - [\alpha]_-}  S^{(i)}_{\alpha}({\bf s}))$ as a formal power series and arguing by induction, starting from $S^{(0)}_{\alpha} = 1$ for all $\alpha$, precisely as in \cite{fgs} section 3.6.
\end{proof}
\begin{corollary}\label{comparison} Fix $c_1, c_2, \bar{\rho} > 0$. There exists $\bar{\lambda} > 0$, depending only on $\bar{\rho}$ and the constants in the support and exponential growth conditions \eqref{support}, \eqref{expBound}, such that for all $\lambda \geq \bar{\lambda}$ the formal power series 
\begin{align*} 
\sum_{\operatorname{disconnected} T} c^{|T|}_1 |\bra \beta , \widetilde{W}_T(Z) \ket| \exp(- c_2 \sum_{v \in T} |Z(\alpha(v))| \lambda) \prod_{v \in T} {\bf s}^{[\alpha(v)]_+ - [\alpha(v)]_-} 
\end{align*}
converges for $|| {\bf s} || < \bar{\rho}$. 
\end{corollary}
\begin{proof}[Proof of Theorem \ref{mainThm}] We show first that, under the assumptions of the Theorem, for all sufficiently large $\lambda$, depending only on $\bar{\rho}$ and the constants in the support condition \eqref{support} and the exponential bound \eqref{expBound} all the formal power series $g(x_{\alpha}, Y(z, Z, \lambda)(x_{\beta}))$ converge absolutely and uniformly for $|| {\bf s} || < \bar{\rho}$.

By our explicit formula \eqref{explicitY} for the action of $Y(z, Z, \lambda)(x_{\beta})$ it remains to prove that there exists $\bar{\lambda} > 0$ as above such that for all $\lambda > \bar{\lambda}$ and $\beta \in \Gamma$ the complex-valued formal power series
\begin{equation}\label{alphaSeries} 
\nonumber\sum_{\operatorname{disconnected} T\,:\, \sum_{v\in T} \alpha(v) = \alpha}  \bra \beta, W_T(Z) \ket H_T(z, Z, \lambda) \prod_{v \in T} {\bf s}^{[\alpha(v)]_+ - [\alpha(v)]_-}
\end{equation}
converges for $|| {\bf s} || < \bar{\rho}$.  

We will in fact prove a statement which is independent of $\alpha$: we claim that there exists $\bar{\lambda} > 0$ as above such that for all $\lambda > \bar{\lambda}$ and $\beta \in \Gamma$ the complex-valued formal power series
\begin{equation*} 
\sum_{\operatorname{disconnected} T }  \bra \beta, W_T(Z) \ket H_T(z, Z, \lambda) \prod_{v \in T} {\bf s}^{[\alpha(v)]_+ - [\alpha(v)]_-}
\end{equation*}
(summing over all decorated trees, without the constraint that $\sum_{v\in T} \alpha(v)$ is fixed) converges for $|| {\bf s} || < \bar{\rho}$. By Proposition \ref{mainEstimateProp} and the comparison principle it is enough to prove the claim for the formal power series 
\begin{equation}\label{comparisonSeries} 
\sum_{\operatorname{disconnected} T } C_1^{|T|} |\bra \beta, \widetilde{W}_T(Z)\ket| \exp( - C_2 \sum_{v \in T} |Z(\alpha(v))| \lambda) \prod_{v \in T} {\bf s}^{[\alpha(v)]_+ - [\alpha(v)]_-}
\end{equation}
for all $\beta$, where $C_1, C_2$ are the constants in \eqref{mainEstimate}. By Corollary \ref{comparison} we can ensure that this converges for $|| {\bf s} || < \bar{\rho}$ by choosing $\bar{\lambda}$ large enough, depending only on $\bar{\rho}$ and \eqref{support}, \eqref{expBound} as required. 

To extend the convergence statement to the matrix elements of the connection $1$-form $A_{\bf s}$ we rely on our explicit formula \eqref{explicitA}. Plugging the expansion for $Y_{\bf s}(x_{\gamma_i})$ in \eqref{explicitA} one checks that each $\Gamma$-graded component of $A_{\bf s}(x_{\gamma_i})$ is given by a finite product of factors which are infinite sums over decorated, disconnected trees and are all dominated by a sum of the form \eqref{comparisonSeries} for possibly larger but fixed constants $C_1, C_2$. 
\end{proof}
\begin{remark}\label{gmnRmk} As we mentioned our proof of Theorem \ref{mainThm} is very much inspired by the work of Gaiotto, Moore and Neitzke in mathematical physics \cite{gmn}. In \cite{gmn} appendix C an integral operator is studied, and the proof of a convergence property for its iterations is sketched using functional analytic techniques. In our present situation we cannot follow this approach directly, since we wish to prove a convergence result that holds \emph{uniformly} as $Z$ approaches the boundary of the strongly generic locus $\Hom^{sg}(\Gamma, \C)$. More precisely the estimate \cite{gmn} (C.20) needed for the contraction property cannot hold uniformly as we approach the boundary $\del \Hom^{sg}(\Gamma, \C)$, since it is based on saddle point approximations such as \cite{gmn} (C.10), (C.11) which do not hold uniformly as $Z \to \del \Hom^{sg}(\Gamma, \C)$. One can get estimates similar to (C.10), (C.11) that depend on the number of vertices of the underlying diagram as in Proposition \ref{mainEstimateProp}, but this is not enough to establish \cite{gmn} (C.20). In the present paper we have replaced the integral operator with the algebraic operator $\mathcal{F}$ acting on formal power series, and proved a convergence result for its iterations for which the type of exponential decay of the functions $H_T(z, Z, \lambda)$ established in Proposition \ref{mainEstimateProp} is sufficient. Proposition \ref{mainEstimateProp} follows in turn from a combination of classical estimates on the Hilbert transform operator, combinatorial considerations, and elementary Sobolev embeddings. Recently C. Garza has informed us of his very interesting work in progress towards proving much stronger results in the functional-analytic framework of \cite{gmn}.
\end{remark}
\section{Application to field theory}\label{gmnSection}
We discuss briefly the original physical setup of \cite{gmn}. In that context one studies the low-energy effective Lagrangian of a class of $\mathcal{N} = 2$ supersymmetric gauge theories on $\R^3 \times S^1_R$ (a circle of radius $R$). This is known to be given by a supersymmetric sigma-model with values in a noncompact hyperk\"ahler fibred manifold $\M \to \B$. The generic fibre is isomorphic to $\Gamma \otimes \R/\Z$, where $\Gamma$ is the lattice of electro-magnetic charges, with a natural skew-symmetric pairing $\bra - , - \ket$. The gauge theory naturally specifies functions on the smooth locus $\mathcal{B}^o \subset \mathcal{B}$ (where the fibres are smooth), the central charge $Z\!: \B^o \to \Gamma^{\vee}\otimes \C$ (which also encodes the energy scale at which we are looking) and the locally constant BPS spectrum $\Omega\!: \B^o \to \Gamma^{\vee}\otimes\Q$. The spectrum $\Omega$ can in fact be realized as the set of Donaldson-Thomas invariants of a $3$CY category $\mathcal{C}$. This is expected from general physical principles (realizing the gauge theory as the field theory limit of a suitable string theory), and was proved mathematically for a large class of theories in \cite{sutherland, bs}.

In \cite{gmn} a set of preferred holomorphic Darboux coordinates for the target metric is found. These coordinates are expressed in terms of a local trivialization of the fibration as formal pairings $\bra \beta, \sum_{\alpha} c_{\alpha} e^{i \th_{\alpha}} \ket$ where $\beta \in \Gamma$ and $\th_{\alpha}$ denotes an angular coordinate on the fibre dual to $\alpha \in \Gamma$. The coefficients $c_{\alpha}$ are functions on $\B^o$ with values in $\Gamma$, given in turn by a sum over trees 
\begin{equation}\label{cSeries} 
c_{\alpha} = \sum_{T\,:\, \sum_v \alpha(v) = \alpha} \frac{1}{|\Aut(T)|}\, \alpha_T G_T(z, Z, R) \prod_{\{v \to w\} \subset T}\bra \alpha(v), \alpha(w)\ket \prod_{v} \dt(\alpha(v), Z)
\end{equation}
where $\dt$ and $\Omega$ are related by \eqref{mobius}. The functions $G_T(z, Z, R)$ are determined explicitly in \cite{gmn} and depend nontrivially on the radius $R$ and a twistor parameter $\zeta \in \C^*$. They are closely related to our $H_T(z, Z, \lambda)$ above. Denoting by $\mathcal{B}^{sg} \subset \mathcal{B}^o$ the locus of generic central charges in $\mathcal{B}^o$, the functions $G_T\!: \C^* \times  \mathcal{B}^{sg} \times \R_{> 0} \to \C^*$ are defined inductively by
\begin{equation}\label{GFun}
G_{T}(\zeta, Z, R) = \frac{1}{2\pi i} \int_{\ell_{\alpha_{T}}}\frac{d w}{w} \frac{w + \zeta}{w - \zeta} \exp(  - R Z(\alpha_{\T}) w^{-1} -   R \bar{Z}(\alpha_{T}) w) \prod_j G_{T_j}(w),
\end{equation}
with the initial condition $G_{\emptyset} = 1$ (recall that with the sign conventions of this paper we have $\ell_{\alpha_T} = \R_{> 0} Z(\alpha_T)$). 

In general the series \eqref{cSeries} contains infinitely many terms. This is because of the symmetry $\dt(\alpha, Z) = \dt(-\alpha, Z)$, expressing the physical fact that every BPS particle of charge $\alpha \in \Gamma$ has a CPT conjugate antiparticle of charge $-\alpha$. In \cite{gmn} no order of summation is specified a priori for  \eqref{cSeries}, so unless the series is absolutely convergent the convergence problem is ill-defined. Following  the arguments of sections \ref{estimatesSec} and \ref{FunSec} verbatim, with the new choice of integration kernel \eqref{GFun}, and in particular recalling that the proof of Theorem \ref{mainThm} gives an estimate on the series \eqref{alphaSeries} which is independent of $\alpha$, we find a corresponding result for the series \eqref{cSeries}.
\begin{corollary}\label{gmnCor} Fix $\z^* \in \C^*$ which does not lie on a distinguished ray. For large enough $R$, independent of $\alpha$, depending only on the support and exponential growth condition, the series \eqref{cSeries} for the $c_{\alpha}$ converges absolutely and uniformly. Moreover there is a common bound $|\bra \beta,  c_{\alpha} \ket| < C$, independent of $\alpha$. It follows that for large enough $R$ the formal expansion $\bra\,-\,, \sum_{\alpha} c_{\alpha} e^{i \th_{\alpha}} \ket$ actually gives a well defined distribution on the torus $\Gamma\otimes\R/\Z$ with values in $\Gamma^{\vee}$.  
\end{corollary}

\end{document}